\documentclass[12pt]{article}
\usepackage[sectionbib]{natbib}
\usepackage{array,epsfig,fancyheadings,rotating}
\usepackage[]{hyperref}  
\usepackage{sectsty, secdot}
\sectionfont{\fontsize{12}{14pt plus.8pt minus .6pt}\selectfont}
\renewcommand{\theequation}{\thesection\arabic{equation}}
\subsectionfont{\fontsize{12}{14pt plus.8pt minus .6pt}\selectfont}

\usepackage{amsmath}
\usepackage{amssymb}
\usepackage{amsfonts}
\usepackage{multirow}
\usepackage{amsthm}

\newtheorem{theorem}{Theorem}
\newtheorem{lemma}{Lemma}
\newtheorem{corollary}{Corollary}

\theoremstyle{definition}

\newtheorem{remark}{Remark}

\newtheorem{condition}{Condition}
\newtheorem{algorithm}{Algorithm}

\usepackage{custom_commands}

\begin{document}


\renewcommand{\baselinestretch}{2}


\markboth{\hfill{\footnotesize\rm José Kling AND Mathias Vetter} \hfill}
{\hfill {\footnotesize\rm On goodness-of-fit testing for self-exciting point processes} \hfill}

\renewcommand{\thefootnote}{}
$\ $\par


\fontsize{12}{14pt plus.8pt minus .6pt}\selectfont \vspace{0.8pc}
\centerline{\large\bf On goodness-of-fit testing for self-exciting point processes}
\vspace{.4cm}
\centerline{José C. F. Kling\(^{1,2}\) and Mathias Vetter\(^1\)}
\vspace{.4cm}
\centerline{\it\(^1\) Christian-Albrechts-Universit\"at zu Kiel, Mathematisches Seminar}
\centerline{\it\(^2\) GEOMAR Helmholtz Centre for Ocean Research}
\vspace{.55cm} \fontsize{9}{11.5pt plus.8pt minus.6pt}\selectfont


\begin{quotation}
    Despite the wide usage of parametric point processes in theory and applications, a sound goodness-of-fit procedure to test whether a given parametric model is appropriate for data coming from a self-exciting point processes has been missing in the literature. In this work, we establish a bootstrap-based goodness-of-fit test which empirically works for all kinds of self-exciting point processes (and even beyond). In an infill-asymptotic setting we also prove its asymptotic consistency, albeit only in the particular case that the underlying point process is inhomogeneous Poisson.

    \vspace{9pt}

    \par
\end{quotation}\par

\def\thefigure{\arabic{figure}}
\def\thetable{\arabic{table}}

\renewcommand{\theequation}{\thesection.\arabic{equation}}

\fontsize{12}{14pt plus.8pt minus .6pt}\selectfont

\section{Introduction}

Self-exciting point processes constitute a rich class of models to describe the evolution of real world events over time. Among the most prominent models in practice are standard Poisson processes (having a constant intensity function), inhomogeneous Poisson processes (those with a deterministic but time-varying intensity process) and Hawkes processes (self-exciting point processes with a conditional intensity which depends on past observations). Such processes are used in a variety of fields such as geology (see e.g.\ \cite{ogat1978} for an early use of Hawkes process models in seismology, but also \cite{dioetal2023} for recent work on inhomogeneous Poisson processes in volcanology), neuroscience (see e.g.\ \cite{burk2006} for neuron models based on inhomogeneous Poisson processes), finance (see \cite{bauhau2009} for an overview on the use of point processes in finance) and telecommunication (see e.g.\ \cite{pinetal2015} and \cite{rizetal2017} for the use of Hawkes processes in social media), among many others.

From a statistical point of view, the situation is usually as follows: One has data coming from a specific class of point processes, i.e.\ observation times $0 < t_1 < t_2 < \ldots$ over some compact interval $[0,T]$, which correspond to the event times of the underlying point process $N$. Statistical procedures then usually involve the estimation of an unknown parameter in the underlying model. For self-exciting point processes it is well known that the distribution of $N$ is uniquely characterized by its conditional intensity process $\la(t)$, i.e.\ the process making \(N(t) - \int_0^t \la(s) ds\) a zero mean locally square integrable martingale with respect to the natural filtration. Hence, parametric models are usually of the form $\la(t) = \nu(t,\te) \text{ for some } \te \in \Te$, where $\Te$ is a finite-dimensional parameter space and $t \mapsto \nu(t,\te)$ in general needs to be a suitably measurable stochastic process indexed by $\te$. The prime interest of the statistician then is to estimate the unknown true parameter $\te_0$.

For self-exciting point processes, but also for certain models based on renewal processes, this is relatively easy as the likelihood function often is directly accessible. Therefore, e.g.\ \cite{ogat1978}, \cite{vere1982} and \cite{choetal1988} were able to explicitly compute maximum likelihood estimators for certain point process models and to prove consistency as well as asymptotic normality, albeit mostly in very restrictive situations. It was only recently that \cite{chehal2013} proved asymptotic normality for the maximum likelihood estimator $\hat \te$ within a large class of parametric models involving self-exciting point processes.

What is lacking, however, is a consistent goodness-of-fit test for parametric point process models. This means, the goal is not to estimate the unknown parameter $\te_0$ or to test certain properties of it, but rather to check whether the proposed parametric model provides a reasonable fit for the observed event times. Formally, what is missing is a sound statistical procedure which tests the null hypothesis that $\la(t) = \nu(t,\te) \text{ for some } \te \in \Te$ against the alternative that there is no $\te \in \Te$ such that $\nu(t,\te)$ fully describes the intensity process.

For independent and identically distributed data, the problem of testing a given parametric model is well-known and dates back to e.g.\ \cite{darl1955}, \cite{kacetal1955} and \cite{lill1967}. The main strategy is always to first estimate the unknown parameter and to then take a suitable distance between two estimators for the unknown distribution function. The first is the parametric distribution function under the model with the unknown parameter replaced by the estimated one, and the second is a non-parametric estimator such as the empirical distribution function. The problem with these approaches, however, is that the (explicit or asymptotic) distribution of this distance is rarely known or depends on the unknown parameter in a complicated way \citep{durb1973}, so that one needs to rely on simulation methods in order to obtain asymptotic quantiles for a valid goodness-of-fit test. Two convincing solutions have been provided somewhat recently by \cite{jograo2004} and \cite{meiswa2007} who have proven under very general assumptions that a bootstrap-based method leads to asymptotically consistent tests.

For simple point processes, however, the picture is much less clear. Both \cite{ogat1988} and \cite{chehal2013} mention that a formal goodness-of-fit test for these models is missing, and they propose the use of an ad-hoc method based on the fact that, if its conditional intensity function is known, then a self-exciting point process can be transformed into a standard Poisson one. To be precise, given observation times $0=t_0 < t_1 < t_2 < \ldots$ and the associated cumulative intensity process \(\La(t) = \int_0^t \la(s) ds\) it is known that the process $0=s_0 < s_1 < s_2 < \ldots$ given by $s_i = \La(t_i)$ is equal in distribution to a standard Poisson process. In particular, the increments $s_i - s_{i-1}$, $i \ge 1$, equal independent standard exponential random variables.

This intuition is used to come up with the afore-mentioned ad-hoc method for goodness-of-fit testing. Under the null hypothesis there exists an unknown true parameter $\te_0$ giving both the true intensity process $\la(t,\te_0)$ as well as the true cumulative intensity process $\La(t, \te_0)$. Clearly, the above time transformation is infeasible as $\te_0$ is unknown, but we can replace $\te_0$ by its consistent estimator $\hat \te$ and in turn obtain an estimator $\hat \La(t)$ as well as event times given by $\hat s_i = \hat \La(t_i)$. The proposed goodness-of-fit tests are then based on testing the null hypothesis whether the increments $\hat s_i - \hat s_{i-1}$, $i \ge 1$, behave like independent standard exponential random variables. This can be done via standard procedures. (A similar approach is proposed for renewal processes; see \cite{beblai1996} and \cite{bebb2013}.)

The problem, however, is the same as historically for i.i.d.\ observations, namely that these procedures do not account for the fact that the parameter (and, hence, the cumulative intensity function) needs to be estimated. Under the null, they treat $\hat s_i - \hat s_{i-1}$ as being i.i.d.\ standard exponential distributed which obviously is not the case. Our main goal in this work therefore is to overcome this issue and to provide a convincing solution to the problem of goodness-of-fit testing for self-exciting point processes. In the spirit of \cite{jograo2004} and \cite{meiswa2007} we establish a novel test which utilizes recent bootstrap techniques for point processes \citep{cavetal2023}. We also provide a thorough asymptotic theory which proves that this new methodology yields a consistent test, albeit only for inhomogeneous Poisson processes.

The paper is organised as follows: In Section \ref{sec:setting} we introduce the formal setting of this work and explain the main idea behind our bootstrap-based goodness-of-fit test. Section \ref{sec:results} contains the statement of the main results of this work, which are the asymptotic consistency of the bootstrap-based test as well as a generalisation of the key theorems in \cite{chehal2013}, and we also discuss the various conditions needed to prove these claims. Sections \ref{sec:simul} and \ref{sec:emp} deal with a thorough simulation study as well as some empirical applications. Finally, all proofs are gathered in Section \ref{sec:proof}.

\section{Setting} \label{sec:setting}

In the following we will work with the setting from \cite{chehal2013} which is specifically tailored to allow for a formal asymptotic theory. If one thinks of the still quite simple class of inhomogeneous Poisson processes, it is clear that in order to check whether a (parametric) model is appropriate or not one needs to potentially have infinitely many observations around every time point $t$ as it is otherwise not obvious to check whether the proposed function $\la$ is reasonable at $t$ or not. A similar reasoning applies to processes with self-exciting components as well. The solution of \cite{chehal2013} to this problem is to work with infill asymptotics; e.g., to assume that, at stage $n$, we are given observations $0=t_0^n < t_1^n < t_2^n < \ldots$ associated with a self-exciting point process $N_n$ over a compact interval, which we choose to be $[0,1]$ from now on. The corresponding intensity process is then given by \(\la_n(t) = a_n \mu(t) + \int_{[0,t)} g(t-s) dN_n(s)\), where $a_n \to \infty$ denotes a sequence of known positive constants and $\mu$ and $g$ are unknown functions (satisfying additional conditions) which represent the deterministic part and the self-exciting part of the intensity process, respectively. Under the null hypothesis the above sequence of intensity processes belongs to a parametric family, i.e.\ it is assumed that there exists some parameter space $\Te$ such that
\begin{equation}
    H_0: \la_n(t) = \la_n(t,\te)
    = a_n \mu(t, \te) + \int_{[0,t)} g(t-s,\te) dN_n(s) \text{ for some } \te \in \Te
\end{equation}
holds true.

Loosely speaking, under the null hypothesis the intensity process (and, hence, the point process itself) follows a shape given by the baseline function $\mu(\cdot, \te_0)$ and the excitation function $g(\cdot,\te_0)$ while the factor $a_n$ allows for a growing number of events as $n \to \infty$ and thus makes an asymptotic theory from observations over $[0,1]$ possible. In \cite{chehal2013} this model is used for asymptotic inference on the true parameter $\te_0$, i.e.\ if $H_0$ holds and $\te_0$ is the true parameter then the authors prove under relatively mild conditions that the maximum likelihood estimator $\hat \te_n$ is consistent for $\te_0$ and asymptotically normal.

Our aim is to provide a consistent goodness-of-fit test for which we utilize the fact that the point process given by $0=s_0^n < s_1^n < s_2^n < \ldots$ with $s_i^n = \La_n(t_i^n,\te_0)$ is equal in distribution to a standard Poisson process. Here, $\La_n(t,\te_0)$ denotes the cumulative intensity process associated with $\la_n(t,\te_0)$. As noted in the introduction, a natural idea is to replace the unknown process $\La_n(t,\te_0)$ by an estimated version $\La_n(t,\hat \te_n)$. This estimated process $\La_n(t,\hat \te_n)$ can then be interpreted as a near optimal fit to the unknown process $\La_n(t,\te_0)$ within the parametric model, and under the null hypothesis the estimated increments $\hat s_i^n - \hat s_{i-1}^n$, $i=1, \ldots, N_n(1)$, based on the transform $\hat s_i^n = \La_n(t_i^n, \hat \te_n)$, should be reasonably close to independent standard exponential random variables.

There exist many tests to check whether a given set of observations was generated by a specific distribution, and in this work we will discuss a version based on empirical Laplace transforms, although many other statistics (including more standard ones like the Kolmogorov-Smirnov approach) would, in principle, work as well. Critical values for the resulting test statistic will then be obtained from a bootstrap procedure explained in detail below. Its validity (i.e.\ consistency under the null hypothesis) will be proven, albeit only for inhomogeneous Poisson processes. For notational convenience we will often drop the index $n$ and simply write e.g.\ $t_i$ instead of $t_i^n$.

In the following, let \(k_n(u,\te,r,s) = \exp\left(-u \int_r^s \la_n(t,\te) dt \right)\), with \(u \ge 0\), and set
\begin{equation*}
    \begin{split}
        L_n(u,\hat \te_n) &= \frac 1{N_n(1)} \sum_{i \ge 1} k_n(u, \hat \te_n, t_{i-1}, t_i) 1_{\{t_i \le 1\}} \\ &= \frac 1{N_n(1)} \sum_{i \ge 1} \exp\left(-u (\La_n(t_i,\hat \te_n) - \La_n(t_{i-1}, \hat \te_n))  \right) 1_{\{t_i \le 1\}}
    \end{split}
\end{equation*}
where $\hat \te_n$ denotes the maximum likelihood estimator for the true parameter $\te_0$ under $H_0$ as discussed in \cite{chehal2013}. By construction, $L_n(u,\hat \te_n)$ plays the role of an empirical Laplace transform based on the estimated increments $\hat s_i^n - \hat s_{i-1}^n$, $i=1, \ldots, N_n(1)$. Hence, under the null hypothesis it should be close to the Laplace transform $L(u) = (1+u)^{-1}$ of a standard exponential variable. As the (random) number of observations $N_n(1)$ is proportional to $a_n$, a reasonable test statistic for $H_0$ is given by \(\sqrt{a_n} \vert \vert G_n \vert \vert\), where $G_n(u) = L_n(u,\hat \te_n) - L(u)$, and \(\vert \vert h \vert \vert^2 = \int_0^\infty h^2(u) \be(u) du\), where \(\be\) is a positive and integrable weight function. Note that
\begin{equation} \label{defHilb}
    \hH = \left\{h: [0,\infty) \to \R ~ \middle\vert~ h \text{ is continuous with } \int_0^\infty h^2(u) \be(u) du < \infty\right\}
\end{equation}
defines a separable Hilbert space with the usual scalar product, \(\langle h_1, h_2 \rangle = \int_0^\infty h_1(u) h_2(u) \be(u) du\), and the associated norm as above.

Recall that $L_n(u,\hat \te_n)$ does not equal the empirical Laplace transform of i.i.d.\ standard exponential variables even under the null hypothesis, as the true cumulative intensity process $\La_n(t,\te_0)$ is unknown and needs to be estimated based on $\hat \te_n$. Hence, asymptotic critical values for $\sqrt{a_n} \vert \vert G_n \vert \vert$ are unknown, and in the spirit of \cite{jograo2004} and \cite{meiswa2007} we propose a bootstrap procedure to estimate these critical values. This procedure mimics how $\sqrt{a_n} \vert \vert G_n \vert \vert$ is obtained and accounts for the fact that the distinction between the true unknown parameter yielding the observations and the estimated parameter used to calculate the time transformation needs to be taken into account.

Precisely, we propose the following strategy: For all $b=1, \ldots, B$, simulate a point process  $\{t_i^{*,b} ~|~ i=0,1,\ldots \}$ according to the intensity process $\la_n(t,\hat \te_n)$ and the corresponding cumulative intensity process $\La_n(t, \hat \te_n)$, for which both variants of the proposed bootstrap strategy in \cite{cavetal2023} can be used. Note specifically that, in the case of inhomogeneous Poisson processes, the simulation mimics the time transformation from above, i.e.\ one simulates a standard Poisson process $0=s_0^{*,b} < s_1^{*,b} < s_2^{*,b} < \ldots$ and then transforms the event times to $t_i^{*,b} = \La_n^{-1}(s_i^{*,b}, \hat \te_n)$. Based on the simulated point process $\{t_i^{*,b} ~|~ i=0,1,\ldots \}$ over $[0,1]$ one then reproduces how the original test statistic is obtained. First, one estimates $\hat \te_n^{*,b}$ from $\{t_i^{*,b} ~|~ i=0,1,\ldots \}$ via maximum likelihood and computes $\hat s_i^{*,b} = \La_n(t_i^{*,b}, \hat \te_n^{*,b})$. Then
\begin{equation*}
    L_n(u,\hat \te^{*,b}) = \frac 1{N_n^{*,b}(1)}  \sum_{i \ge 1} \exp\left(-u (\La_n(t_i^{*,b}, \hat \te_n^{*,b}) - \La_n(t_{i-1}^{*,b},\hat \te_n^{*,b}))  \right) 1_{\{t_i^{*,b} \le 1\}}
\end{equation*}
and \(G_n^{*,b}(u) = L_n(u,\hat \te^{*,b}) - L(u)\) are computed in exactly the same way as $L_n(u,\hat \te_n)$ and $G_n(u)$. Finally, one uses the empirical quantiles of $\sqrt{a_n} \vert \vert G_n^{*,b} \vert \vert$, $b=1, \ldots, B$, as critical values which should approximate the unknown true critical values of $\sqrt{a_n} \vert \vert G_n \vert \vert$  as $B \to \infty$. Note that a feasible test statistic can be constructed easily without the knowledge of $a_n$. 

To summarize, the only difference regarding the generation of $\sqrt{a_n} \vert \vert G_n \vert \vert$ and the ones of $\sqrt{a_n} \vert \vert G_n^{*,b} \vert \vert$, $b=1, \ldots, B$, regards the distribution of the underlying point process. Under the null hypothesis the parameter is $\te_0$ for $\sqrt{a_n} \vert \vert G_n \vert \vert$, whereas the bootstrapped processes are generated according to $\hat \te_n$ (which estimates $\te_0$ consistently). Intuitively it is thus clear that this bootstrap strategy leads to asymptotically valid critical values if a certain continuity property of the distribution of $\sqrt{a_n} \vert \vert G_n \vert \vert$ with respect to the parameter $\te$ generating the original point process is satisfied. This will be made precise below.

\section{Results} \label{sec:results}

Proving consistency of the bootstrap procedure equals showing that the empirical quantiles obtained above converge (in a certain sense) to the unknown asymptotic quantiles of $\sqrt{a_n} \vert \vert G_n \vert \vert$ if the null hypothesis holds. Hence, from now on we assume that there exists $\te_0 \in \Te$ such that $\la_n(t)
= a_n \mu(t, \te_0) + \int_{[0,t)} g(t-s,\te_0) dN_n(s)$ holds true. A natural goal is to prove
\begin{equation} \label{convboot}
    \lim_{n \to \infty, B \to \infty} \P \left(\sqrt{a_n} \vert \vert G_n \vert \vert \ge (F_n^B)^{-1}(1-\al) \right) = \al
\end{equation}
for any $\al \in (0,1)$, where \(F_n^B(x) = \frac 1B \sum_{b=1}^B 1_{\{ \sqrt{a_n} \vert \vert G_n^{*,b} \vert \vert \le x\}}\), and $(F_n^B)^{-1}$ denotes its generalized inverse.

In order to deduce (\ref{convboot}) we will rely on previous results from \cite{buekoj2019} and \cite{vand1998}. Let us start with some additional notation: Given is a series of experiments $(\xX_n, \bB_n, \pP_n)$ with $\pP_n = \{P_{\te,n} ~|~ \te \in \Te \}$ for some parameter space $\Te \subset \R^d$ and we regard any $x_n \in \xX_n$ as a sample of the original point process. Associated with these samples is the maximum likelihood estimator $\hat \te_n = f_n(x_n)$ as well as $\sqrt{a_n} \vert \vert G_n \vert \vert = k_n(x_n)$ whose distribution we call $V_{\te,n}$. We will later work under conditions which grant consistency and asymptotic normality of the estimator $\hat \te_n$ for the unknown true parameter $\te_0$, invoking the results from \cite{chehal2013}. See in particular the more general Lemma \ref{lem1}(b) later.

Now, Lemma 4.2 in \cite{buekoj2019} states a sufficient condition to deduce (\ref{convboot}), and it boils down to checking two separate conditions. On one hand, one needs to show that their Condition 4.1 holds true, i.e.\ for any fixed $\te_0$ one needs weak convergence of $\sqrt{a_n} \vert \vert G_n \vert \vert$ towards a limiting random variable with a continuous distribution function. In other words, setting $d$ as the bounded Lipschitz metric which metrizes weak convergence, one has to show
\begin{equation}\label{Gnweak}
    \lim_{n\to \infty}d\left(V_{\te_0,n},V_{\te_0}\right) = 0 \quad \text{for all } \te_0 \in \Te,
\end{equation}
where $V_{\te_0}$ is a distribution on $\R$ with a continuous distribution function. The other condition is any of the equivalent statements in their Lemma 2.2, part (c) being the one used here. In accordance with the notation from above, a generic bootstrap sample is based on simulating according to $P_{\hat \te_n,n}$, and we obtain both $\hat \te_n^{*}$ and $\sqrt{a_n} \vert \vert G_n^{*} \vert \vert$ by applying the functions $f_n$ and $g_n$ to the bootstrap sample. We hence denote the conditional distribution of $\sqrt{a_n} \vert \vert G_n^{*} \vert \vert$, as it is based on the estimator $\hat \te_n$, by $V_{\hat \te_n,n}$, and Lemma 2.2(c) then becomes proving \(d\left(V_{\hat \te_n,n},V_{\te_0}\right) \stackrel{P_{\te_0,n}}{\longrightarrow} 0\) for all \(\te_0 \in \Te\).

Setting \(\Phi_{\te_n,n}(h) = d\left(V_{\te_0 + h/\sqrt{a_n},n},V_{\te_0}\right)\), an equivalent formulation is \(\Phi_{\te_n,n}(\sqrt{a_n}(\hat \te_n - \te_0))  \stackrel{P_{\te_0,n}}{\longrightarrow} 0\) for all \(\te_0 \in \Te\), and, in the spirit of Exercise 23.5 in \cite{vand1998}, a sufficient condition to deduce this convergence is to show
\begin{equation}\label{Gnweakcond2}
    \lim_{n \to \infty} d\left(V_{\te_0 + h_n/\sqrt{a_n},n},V_{\te_0}\right) = 0 \quad \text{for all } \te_0 \in \Te, \text{ and all } h_n \to h,
\end{equation}
or equivalently $\lim_{n \to \infty} \Phi_{\te_n,n}(h_n) = \Phi(h) \equiv 0$ for all $\te_0 \in \Te$ and all $h_n \to h.$ Indeed, utilizing the asymptotic normality $\lim_{n\to \infty} d(\sqrt{a_n}(\hat \te_n - \te_0),N_{\te_0}) = 0$ with $N_{\te_0} \sim \nN(0,\iI(\te_0)^{-1})$ (Theorem 2 in \cite{chehal2013}) and the extended continuous mapping theorem (Theorem 18.11 in \cite{vand1998}), for any $\te_0 \in \Te$ we then can first conclude the weak convergence of 
$\Phi_{\te_n,n}(\sqrt{a_n}(\hat \te_n - \te_0))$ towards the constant $\Phi(N_{\te_0}) = 0$ and then convergence to zero in $P_{\te_0,n}$-probability afterwards.

Hence, besides checking the conditions for an application of Theorem 2 in \cite{chehal2013}, what remains in order to prove (\ref{convboot}) is to show (\ref{Gnweak}) and (\ref{Gnweakcond2}). In fact, we will prove the stronger statement
\begin{equation} \label{Gnweakcond3}
    \lim_{n\to \infty}d\left(V_{\te_n,n},V_{\te_0}\right) = 0 \quad \text{for all } \te_0 \in \Te, \text{ and all } \te_n \to \te_0
\end{equation}
from which both assertions obviously follow.

Before we come to the main result of our work, we will give some conditions and state a few auxiliary claims which generalize some results from \cite{chehal2013} and might be of independent interest. Note that while the strategy of proving (\ref{Gnweakcond3}) works for all self-exciting processes, our main restriction will be that a full proof is only provided for inhomogeneous Poisson processes in the sense below.

\begin{condition} \label{condmain} Suppose the following conditions to hold:
    \begin{itemize}
        \item[(a)] $\Te \subset \R^d$ is a compact parameter space whose interior is connected and contains a $d$-dimensional open ball which contains the limiting (true) parameter $\te_0$.
        \item[(b)]  $a_n \to \infty$ is a known sequence, $\te_n \to \te_0$ denotes a sequence of unknown parameters, and for each $n$ a point process $0 = t_0^n < t_1^n < t_2^n < \ldots$ is observed over $[0,1]$ according to the intensity function \(\la_n(t) = \la_n(t,\te_n) = a_n \mu(t,\te_n)\) and the associated cumulative intensity function \(\La_n(t) = \La_n(t,\te_n) = \int_0^t a_n \mu(s,\te_n) ds\). We also set $C(t,\te_n) = \int_0^t \mu(s,\te_n) ds.$
        \item[(c)] The function $(t,\te) \mapsto \mu(t,\te)$ is uniformly bounded away from zero, i.e., \(\inf_{t \in [0,1], \te \in \Te} \vert \vert \mu(t,\te) \vert \vert > 0\), and it is two times continuously differentiable with uniformly bounded zeroth, first and second derivatives.
        \item[(d)] $\be: [0,\infty) \to [0,\infty)$ is a known integrable function and defines a Hilbert space $\hH$ as in (\ref{defHilb}).
    \end{itemize}
\end{condition}

The first lemma essentially generalizes results from \cite{chehal2013}. As mentioned before, we will work under conditions which grant consistency and asymptotic normality of $\hat \te_n$ as an estimator for $\te_0$ whenever the latter is the true unknown parameter. However, for a proof of the bootstrap procedure to work we need to discuss the slightly more general situation stated in Condition \ref{condmain}, namely the one where the true parameter at stage $n$ is $\te_n$, and these parameters do not have to equal $\te_0$ but only have to converge to it. In this more general case, which obviously contains the case of a constant parameter, we will need a version of asymptotic normality (and, hence, consistency) as well.

This result will be proved in the more general situation of truly self-exciting processes. Hence, we need a more comprehensive set of assumptions as well as additional notation.

\begin{condition}\label{condSEPP}
    Assume Condition \ref{condmain} with item (b) modified to
    \begin{itemize}
        \item [(b)]  $a_n \to \infty$ is a known sequence, $\te_n \to \te_0$ denotes a sequence of unknown parameters, and for each $n$ a point process $0 = t_0^n < t_1^n < t_2^n < \ldots$ is observed over $[0,1]$ according to the intensity process
            \begin{equation*}
                \la_n(t) = \la_n(t,\te_n) = a_n \mu(t,\te_n) + \int_0^t g(t-u,\tn) dN_n(u)
            \end{equation*}
            and the associated cumulative intensity process
            \begin{equation*}
                \La_n(t) = \La_n(t,\te_n) = \int_0^t \left( a_n \mu(u,\te_n) + \int_0^u g(u-s,\tn) dN_n(s) \right) du.
            \end{equation*}
            Note that the mean intensity process associated with $\lambda_n(t, \te)$ satisfies the linear Volterra equation \(\eta_n(t, \te) = a_n \mu(t, \te) + \int_0^t g(t-u, \te) \eta_n(u, \te)du\). We then define $c(t,\te)$ via \(a_n c(t,\te)  = \eta_n(t,\te)\) and also write $C(t,\te) = \int_0^t c(u,\te) du.$
    \end{itemize}

    We also additionally assume:
    \begin{itemize}
        \item[(e)] The function $(t,\te) \mapsto g(t,\te)$ is two times continuously differentiable with uniformly bounded zeroth, first and second derivatives.
        \item[(f)] The matrix-valued function
            \begin{equation}
                \mathcal{I}(\te) = \int_0^1 \frac{ \left( \pt \mu(t,\te) + \int_0^t \pt g(t-u,\te) c(u,\te) du \right)^{\otimes 2}}{\mu(t,\te) + \int_0^t g(t-u,\te) c(u,\te) du}
            \end{equation}
            is non-singular at the true parameter value $\te_0$. In the $d$-dimensional situation we use the notation $\partial_\te = \partial/\partial \te$ and $\partial_\te^{\otimes 2} = \partial^2/\partial \te \partial \te^T$ in accordance with \cite{chehal2013}.
    \end{itemize}	
\end{condition}

Notice that an inhomogeneous Poisson process is a degenerate case of a self-exciting process where \(g = 0\), therefore any inhomogeneous Poisson process satisfying Condition \ref{condmain} will automatically satisfy Condition \ref{condSEPP}.

Before we state the first lemma, we finally give a preparation which is borrowed from the proof of Theorem 2 in \cite{chehal2013} and provides a typical representation for $\sqrt{a_n}(\hat \te_n - \te_n)$ when maximum likelihood estimation is used. For \(S_n(\te) = \int_0^1 U_n(t,\te) dM_n(t,\te)\) with
\begin{equation} \label{defUn}
    U_n(t,\te) = \frac{\pt \mu(t,\te) + \int_0^t \pt g(t-u,\te) da_n^{-1} N_n(u)}{\mu(t,\te) + \int_0^t g(t-u,\te) da_n^{-1} N_n(u)}
\end{equation}
and
\begin{equation}
    M_n(t,\te) = N_n(t) - \int_0^t \left( a_n\mu(u,\te) +\int_0^u g(u-s,\te) dN_n(s)  \right) du
\end{equation}
we have
\begin{equation} \label{propSn}
    0 = a_n^{-1/2} S_n(\hat \te_n) = a_n^{-1/2} S_n(\te_n) + a_n^{-1} \pt S_n(\te)|_{\te = \Wte_n} a_n^{1/2} \left(\hat \te_n - \te_n\right)
\end{equation}
for some intermediate $\Wte_n$.

\begin{lemma} \label{lem1} Suppose that Condition \ref{condSEPP} holds and recall that $N_n(1)$ denotes the random number of events over $[0,1]$.
    \begin{itemize}
        \item[(a)] We have \(\sup_{t \in [0, 1]} |\ani N_n(t) - C(t, \tz)| \pn 0\) as well as \(N_n(1)/\ka_n \pn 1\) and \(\E\left( N_n(1) \right)/\ka_n \longrightarrow 1\), with $\ka_n = a_n C(1,\te_0)$.
        \item[(b)] Let $\hat \te_n$ denote the maximum likelihood estimator for $\te_n$ based on the observations $0=t_0^n < t_1^n < t_2^n < \ldots$ over $[0,1]$. Then \(\sqrt{a_n} \left(\hat \te_n - \te_n \right) \tol \nN \left(0,\iI(\te_0)^{-1}\right)\), with $\iI(\te_0)$ as in Condition \ref{condSEPP}(f).
        \item[(c)] With $\iI(\te_0)$ as above we have \(a_n^{-1} \pt S_n(\te)|_{\te = \Wte_n} \pn -\iI(\te_0)\).
    \end{itemize}
\end{lemma}


We are now in a position to formulate the main theorem of this work. As noted before, and in contrast to Lemma \ref{lem1}, it will only be stated for inhomogeneous Poisson processes in the sense of Condition \ref{condmain}.

\begin{theorem} \label{thm1}
    Let
    \begin{equation*}
        k_n(u,\te,r,s) = \exp\left(-u \int_r^s a_n \mu(t,\te) dt \right), \quad u\ge 0,
    \end{equation*}
    and let $\hat \te_n$ be the maximum likelihood estimator for $\te_n$ based on the observations $0=t_0^n < t_1^n < t_2^n < \ldots$ over $[0,1]$. Set
    \begin{equation*}
        L_n(u,\hat \te_n) = \frac 1{N_n(1)} \sum_{i \ge 1} k_n(u, \hat \te_n, t^n_{i-1}, t^n_i) 1_{\{t^n_i \le 1\}}, \quad L(u) = (1+u)^{-1}.
    \end{equation*}
    Then, under Condition \ref{condmain}, $\sqrt{a_n} \vert \vert G_n \vert \vert$ with $G_n(u) = L_n(u,\hat \te_n) - L(u)$ converges weakly to a limiting distribution $V_{\te_0}$ with a continuous distribution function.
\end{theorem}

As a direct consequence of Lemma \ref{lem1} and Theorem \ref{thm1} we obtain consistency of the bootstrap procedure.

\begin{corollary}
    Under Condition \ref{condmain} we have
    \begin{equation*}
        \lim_{n \to \infty, B \to \infty} \P \left(\sqrt{a_n} \vert \vert G_n \vert \vert \ge (F_n^B)^{-1}(1-\al) \right) = \al
    \end{equation*}
    for any $\al \in (0,1)$.
\end{corollary}

\begin{remark}
    \begin{itemize}
        \item[(a)] As noted before, to base a goodness-of-fit test on Laplace transforms is only one of many options, as essentially all methods to test the difference between two distributions are in play. We use this specific method essentially only for the reason that it allows for easy differentiability with respect to various parameters, a fact which is heavily exploited in the proof of Theorem \ref{thm1}. Empirically, there is no problem in choosing other variants such as the Kolmogorov-Smirnov distance which performs equally well in simulations.
        \item[(b)] We have also conducted further simulations which check that the procedure is not only valid for the class of (inhomogeneous) Poisson processes, but it also works for other self-exciting processes such as Hawkes processes (which are explicitly discussed in Section \ref{sec:simul}), Hawkes processes with varying baseline intensity, renewal processes, and the two-stage model proposed in \cite{Selva2022}. In all cases, the general procedure for the bootstrap-based test is the same, with the differences only residing in the specific algorithms for estimating the unknown parameter and in the calculation of the transform \(t_i \mapsto s_i = \Lambda(t_i)\).
				\item[(c)] While the main strategy of the proof of Theorem \ref{thm1} as laid out in Section \ref{subsect:pth1} should remain valid for general self-exciting point processes, we heavily rely on a variety of properties of inhomogeneous Poisson processes when it comes to detailed derivations. In particular, this regards distributional aspects in the proof of Lemma \ref{lem3} as well as a simplified form of $U_n$ from (\ref{defUn}) which is exploited in the proof of Theorem \ref{thm2} below. Finding an alternative proof which also works e.g.\ for Hawkes processes is an important topic for future research. 
    \end{itemize}
\end{remark}

\section{Simulation study}\label{sec:simul}

This section aims to assess the performance of our bootstrap-based goodness-of-fit test by investigating its type I error and power in different settings and comparing it with a popular alternative method.

\cite{ogat1988} presented several procedures based on residual analysis to assess the fit of a model to the data. One of them relies on the fact that, if \(t_1, \dots, t_n\) are event times originating from a process \(N(t)\) with corresponding conditional intensity process \(\lambda(t)\), then the random variables \(s_1, \dots, s_n\), defined via \(s_i = \int_0^{N(t_i)} \lambda(s) ds\), are equal in distribution to the event times of a Poisson process. Therefore, a good strategy would be simply to test if the interevent times of the time transformed process follow a unit exponential distribution.

In the situation of goodness-of-fit testing, however, where the intensity process \(\lambda\) depends on a parameter \(\theta\), the transformation \((t_1, \dots, t_n) \to (s_1, \dots, s_n)\) can only be calculated using an estimator \(\hat{\theta}\) for the true parameter \(\theta_0\). \cite{bebb2013} provides an implementation of such a test, which will be used here as a reference for the performance of our method and to show to which extent parameter estimation error can affect the results.

\subsection{Type I error}\label{sub:typeI}

In this subsection we discuss the type I error for both methods and check to which percentage the null hypothesis was incorrectly rejected. The following procedure was carried out using a level of $0.05$.

\begin{algorithm}\label{alg:test_typeI}
    \begin{enumerate}
        \item Simulate 1000 data sets from a known process with fixed parameters.
        \item Perform the tests for each simulation, correctly assuming that the null hypothesis is true.
        \item Compute the percentage of times the null hypothesis was rejected with level 0.05.
    \end{enumerate}
\end{algorithm}

Figure \ref{fig:typeI_poisson} shows the result of six applications of Algorithm \ref{alg:test_typeI} for Poisson processes on the interval \([0, 1]\), each corresponding to different parameters. We can see that our bootstrap-based test, when performed with \(1000\) bootstrap samples, maintains a level close to \(0.05\) for processes with an expected number of events of as few as 30. The non-bootstrap-based test, in contrast, is not consistent as it is very conservative and rarely rejects the null hypothesis. This behaviour might also point towards a potential lack of power.

\begin{figure}
    \begin{center}
        \includegraphics[width=0.5\textwidth]{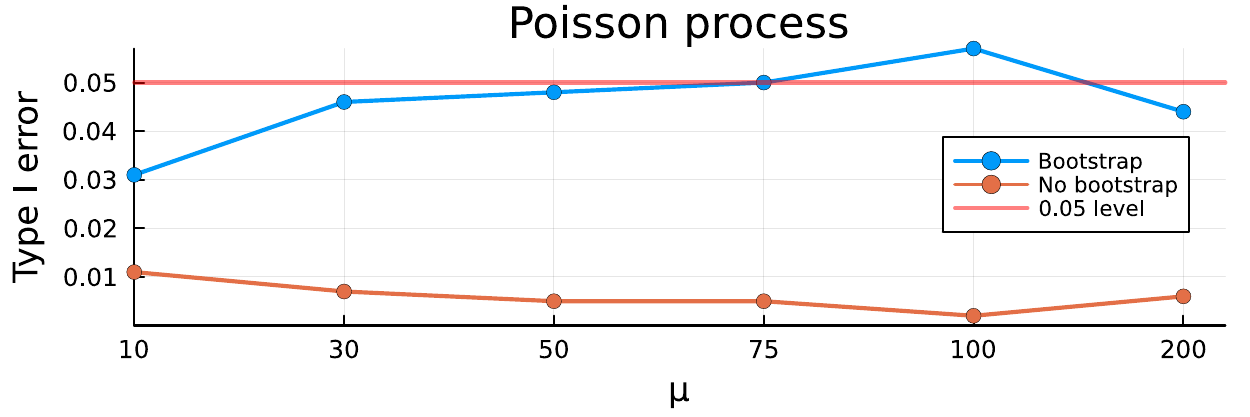}
        \caption{Type I error for the bootstrap and non-bootstrap based tests for Poisson processes with different parameters.}
        \label{fig:typeI_poisson}
    \end{center}
\end{figure}

This pattern persists also for more general Hawkes processes, which are considerably more complex. For the next tests, we consider Hawkes processes with an exponential activation function. Figure \ref{fig:typeI_hawkes} shows the results of the procedure for several parameter configurations. Hawkes processes have three parameters: \(\mu\), the constant baseline intensity; \(\alpha\), the amplitude of the increase in intensity following each event; and \(\beta\), the decay rate. For a realization \((t_1 < \dots < t_n)\) of a Hawkes process, its conditional intensity function is given by \(\lambda(t; \mu, \alpha, \beta) = \mu + \sum_{t_i < t} \alpha e^{-\beta t_i}\).

Of particular importance is the branching factor \(\alpha / \beta\), which is the expected number of events a single event causes due to the self-exciting component. In Figure \ref{fig:typeI_hawkes}, the parameter \(\mu\) and the ratio \(\alpha / \beta \) were kept fixed in each subfigure. By increasing the parameter \(\alpha\) the response of the intensity function to an event happens faster and therefore the behaviour is less similar to a Poisson process.

\begin{figure}
    \begin{center}
        \includegraphics[width=0.9\textwidth]{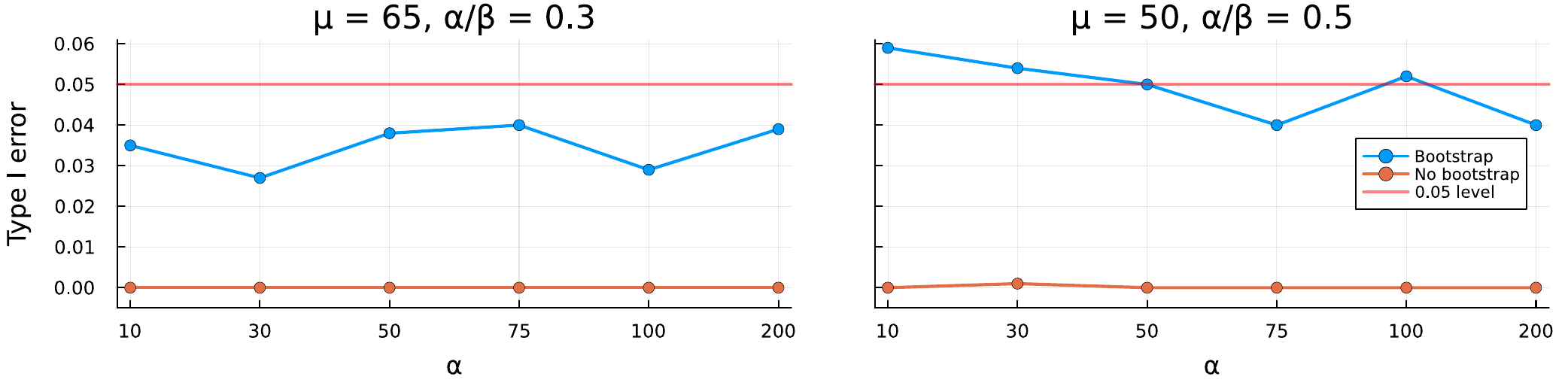}
        \caption{Type I error for both tests with respect to Hawkes processes.}
        \label{fig:typeI_hawkes}
    \end{center}
\end{figure}

In all tests, the expected number of events is approximately equal, and we can see that a reasonably good level is maintained for all configurations. This behaviour is in sharp contrast to the performance of the test without bootstrap. 

\subsection{Power}

The next step is to test the power of the methods. The procedure is similar to Algorithm \ref{alg:test_typeI}, but now two distinct models \(A\) and \(B\) are needed. The idea is to test how effective the tests are in rejecting incorrect hypotheses.

\begin{algorithm}\label{alg:test_power}
    \begin{enumerate}
        \item Simulate 1000 data sets from a known process \(A\) with fixed parameters.
        \item Perform the tests for each simulation, incorrectly assuming that the data stems from some process $B$.
        \item Compute the percentage of times the null hypothesis was rejected with level 0.05.
    \end{enumerate}
\end{algorithm}

In Figure \ref{fig:power_comparison}, the power is plotted for a Hawkes process as model \(A\) and a Poisson process as model \(B\). In essence, this graphic shows how good each of the tests are in differentiating Poisson from Hawkes processes.

\begin{figure}
    \begin{center}
        \includegraphics[width=1\textwidth]{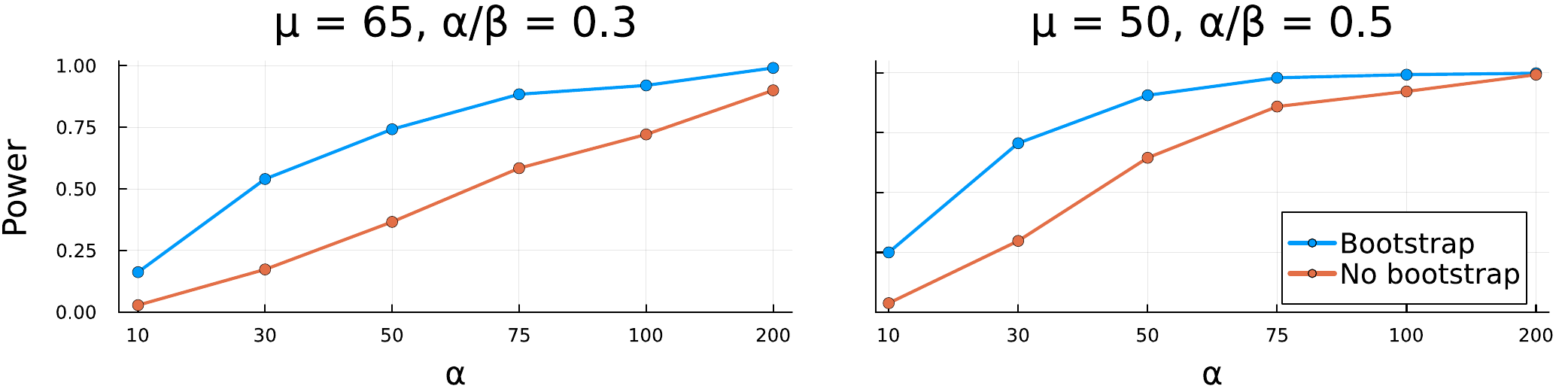}
        \caption{Power of tests for simulated Hawkes processes with the null hypothesis that these are Poisson processes.}
        \label{fig:power_comparison}
    \end{center}
\end{figure}

It can be seen that both tests perform better for larger values of the parameter \(\alpha\) and the branching factor \(\alpha / \beta\), but the bootstrap-based approach is clearly superior.

\section{Empirical application} \label{sec:emp}

It is widely accepted that volcanism can affect local and global climate, primarily through the injection of aerosols into the stratosphere \citep{Robock2000}. A less known and still being explored hypothesis is that climate variations may influence volcanism. The proposed mechanism for this effect is the stress change in Earth's crust due to variations in sea level and thickness of the ice sheets (e.g., \cite{Rampino1979}; \cite{Jellinek2004}; \cite{Cooper2018}).

\cite{Kutterolf2013} presents a compilation of eruption records from multiple sites around the Pacific Ring of Fire over the last 1 million years. These records are based on tephra layers from cores drilled on the ocean floor which are identified and dated. Younger layers tend to suffer less from degradation than older ones, possibly causing an artificial peak of volcanic activity. To mitigate this bias, eruptions dated within the first 50 thousand years were excluded from the analysis. This potential bias can be seen in Figure \ref{fig:records}, where we show the entire eruption records smoothed with a Gaussian kernel.

\begin{figure}
    \begin{center}
        \includegraphics[width=0.9\textwidth]{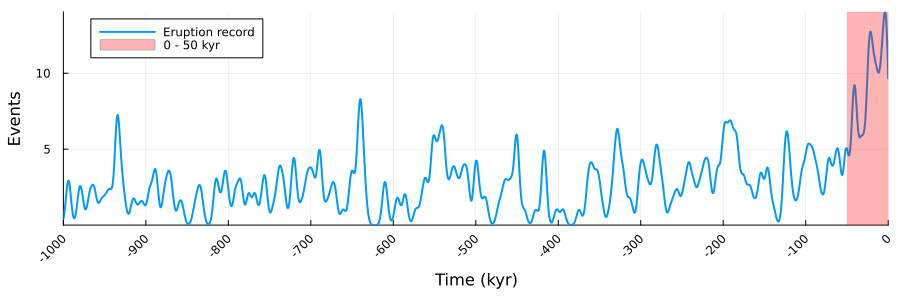}
        \caption{Eruption record smoothed with a Gaussian kernel (\(\sigma\) = 3000) and highlighted period excluded from the analysis.}
        \label{fig:records}
    \end{center}
\end{figure}

To determine whether glacial cycles affect the eruptive process, the data set was split into two groups based on the sites' latitudes. The first part contains eruptions from tropical regions that were not glaciated (Peru, Ecuador, Central America, Philippines, and Tonga). The second comprises sites in extratropical sites in both hemispheres which were more likely subjected to glaciation (New Zealand, Nankai Trough, Kamchatka, Alaska, and the Aleutian Basin and Arc).

Applying our bootstrap-based goodness-of-fit test to both data sets with the hypothesis that they are distributed as Poisson processes shows distinctively different patterns. While for the tropical sites the returned p-value is 0.44, suggesting that a Poisson process is a good fit for the data, the p-value for the extratropical sites is 0.03. For the hypothesis that the extratropical eruptions stem from a Hawkes process as in Section \ref{sub:typeI}, in contrast, the p-value of the test is 0.80. The estimated conditional intensity functions for both data sets are shown in Figure \ref{fig:CIFs}.

\begin{figure}
    \begin{center}
        \includegraphics[width=0.9\textwidth]{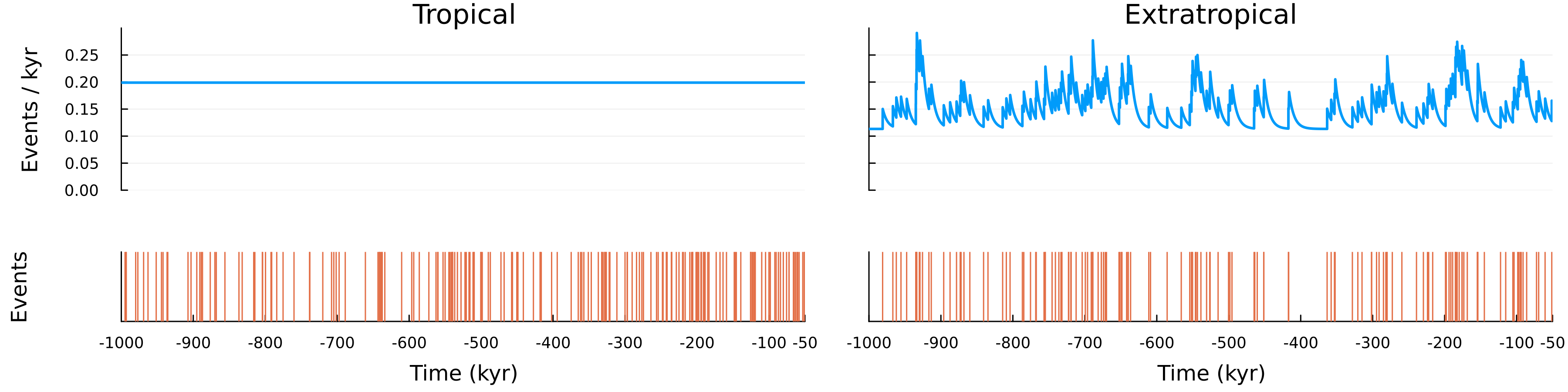}
        \caption{Estimated CIF for tropical (left) under the Poisson assumption and extra\-tropical (right) sites using a Hawkes process.}
        \label{fig:CIFs}
    \end{center}
\end{figure}

These results show that, while tropical eruptions occurred uniformly during the period, in the extratropical regions there is a tendency for events to cluster in time, suggesting different mechanisms governing the processes. Note that the non-bootstrap-based test from Section \ref{sec:simul}, when applied to the hypothesis that the extratropical data set is distributed as a Poisson process, returns a p-value of 0.42, in clear disagreement with the result for the bootstrap-based test. This might be a consequence of the non-bootstrap based being very conservative as discussed previously.

\section{Proofs} \label{sec:proof}

Throughout the proofs we will use $K$ as an unspecified positive constant which might change from line to line.

\subsection{Proof of Lemma \ref{lem1}}

We will start with a few auxiliary results for which our notation mostly follows the one in \cite{chehal2013}. We assume Condition \ref{condSEPP} in all cases.

\begin{lemma}\label{hbound}
    The function 	\(c(t, \te)\) is uniformly bounded in $t\in[0,1]$ and $\te \in \Te$.
\end{lemma}
\begin{proof}
    By definition, $c(t,\te)$ satisfies the equation
    \begin{equation} \label{formh}
        c(t, \te) = \mu(t, \te) + \int_0^t g(t-u, \te) c(u, \te)du
    \end{equation}
    and is hence continuous as a function in $t$. Then 
    \begin{equation*}
        c(t, \te) \le K + K \int_0^t c(u, \te) du  \leq K(1+\exp(K))
    \end{equation*}
		by the uniform boundedness of $\mu$ and $g$ and Gronwall's lemma.
\end{proof}

\begin{lemma}\label{hconv}
    We have \(\sup_{t \in [0, 1]} |c(t, \tn) - c(t, \tz)| \longrightarrow 0\).
\end{lemma}
\begin{proof}
    Using (\ref{formh}) we obtain
    \begin{equation*}
        \begin{split}
            \left| c(t, \tz) - c(t, \tn) \right| \le& \left| \mu(t,\tz) - \mu(t,\tn) \right| +  \left| \int_0^t \left( g(t-u,\tz) - g(t-u,\tn) \right) c(t,\tn) du \right|\\
                                                    & + \left| \int_0^t g(t-u,\tz) \left( c(t,\tz) - c(t,\tn) \right)du \right|.
        \end{split}
    \end{equation*}
    Recall that $\mu$ and $g$ are continuously differentiable with uniformly bounded derivatives. Hence with Lemma \ref{hbound} we have
    \begin{equation*}
        \sup_{t \in [0,1]} \left| \mu(t,\tz) - \mu(t,\tn) \right| +  \left| \int_0^t \left( g(t-u,\tz) - g(t-u,\tn) \right) c(t,\tn) du \right| \le K(\tn - \tz).
    \end{equation*}
    Again from Gronwall's lemma we can conclude
    \begin{equation*}
        \left| c(t, \tz) - c(t, \tn) \right| \le K(\tn - \tz) + K \int_0^t \left| c(t,\tz) - c(t,\tn) \right| du \le K(\tn - \tz)
    \end{equation*}
    for all \(t \in [0,1]\), from which the claim follows.
\end{proof}

We are now in a position to prove Lemma \ref{lem1}(a). For the first claim, one can use the inequality
\begin{equation*}
    \left|\ani N_n(t) - C(t, \tz) \right| \le \left| \ani N_n(t) - C(t, \tn) \right| + \left| C(t, \tn)- C(t, \tz)\right|,
\end{equation*}
and deduce \(\sup_{t \in [0, 1]} | \ani N_n(t) - C(t, \tn)| \pn 0\) along the lines of Theorem 1 from \cite{chehal2013} by replacing \(\tz\) by \(\tn\) in all occurrences, and redefining \(\bar{\mu}\) and \(\bar{g}\) as
\begin{equation*}
    \bar{\mu} = \sup_{\substack{t \in [0,1]\\\theta \in \Theta}} \mu(t, \te) \quad \text{ and } \quad \bar{g} = \sup_{\substack{t \in [0,1]\\\theta \in \Theta}} g(t, \te).
\end{equation*}
Furthermore,
\begin{equation*}
    \sup_{t \in [0, 1]} \left|C(t, \tn) - C(t, \tz)\right| = \sup_{t \in [0, 1]} \left|\int_0^1 c(u, \tn) - c(u, \tz) du \right| \longrightarrow 0
\end{equation*}
is a consequence of Lemma \ref{hconv}.

It is then easy to deduce
\begin{equation*}
    \left|\frac{N_n(1)}{\ka_n} - 1 \right| = \frac{\left|N_n(1)-a_n C(1,\te_0) \right|}{a_n C(1,\te_0)} = \frac{\left|\ani N_n(1)-C(1,\te_0) \right|}{C(1,\te_0)} \pn 0
\end{equation*}
from the fact that $\mu$ is uniformly bounded from below, as well as
\begin{equation*}
    \frac{\E\left(N_n(1) \right)}{\ka_n} -1  = \frac{\int_0^1 a_n \left(c(t,\tn)-c(t,\tz) \right) dt}{a_n \int_0^1 c(t,\tz) dt} = \frac{\int_0^1 \left(c(t,\tn)-c(t,\tz) \right) dt}{\int_0^1 c(t,\tz) dt} \longrightarrow 0
\end{equation*}
utilizing that $\eta(t,\tn) = a_n c(t,\tn)$ is the mean intensity function of the point process.

For the remaining parts of Lemma \ref{lem1} we first state and prove a few corollaries from the previous results. The first is a direct consequence of Lemma \ref{hconv} and uniform boundedness of $g$ and its derivatives and hence stated without proof.

\begin{corollary}\label{gconv} We have
    \begin{equation*}
        \sup_{\substack{t \in [0,1]\\\te \in \Theta}} \left| \int_0^t g(t-u, \te) c(t,\tn)du - \int_0^t g(t-u, \te) c(t,\tz)du \right| \conv 0,
    \end{equation*}
    and similarly with $g$ replaced by $\pt g$ or \(\ptt g\).
\end{corollary}

\begin{corollary}\label{lambdaconv}
    For every fixed \(\te \in \Theta\) we have
    \begin{equation*}
        \sup_{t \in [0, 1]} \left| \int_0^t g(t-u, \te) d \ani N_n(u) - \int_0^t g(t-u, \te) c(u,\tn) du \right| \pn 0,
    \end{equation*}
    and similarly for other functions with uniformly bounded derivatives. 	
\end{corollary}
\begin{proof}
    By using integration by parts, we get
    \begin{equation*}
        \begin{split}
            & \int_0^t g(t-u, \te) (da_n^{-1}N_n(u) - c(u, \tn)du) \\
            & ~~~~~~~~= (\ani N_n(t) - C(t, \tn)) g(0,\te) + \int_0^t (\ani N_n(u) - C(u, \tn))  g'(t-u, \te) du\\
        \end{split}
    \end{equation*}
    where \(g'\) denotes the derivative of \(g\) with respect to \(t\). The claim can then be deduced easily from Lemma \ref{lem1}(a) and uniform boundedness of all functions involved.
\end{proof}

In order to prove the remaining two claims in Lemma \ref{lem1}, we need to first establish the consistency of \(\thn\) as an estimator for \(\tz\). For this proof we will use Theorem 5.9 in \cite{vand1998}, with the following additional notation of
\begin{equation*}
    \begin{split}
        s_n(\te) = & \int_0^1 \frac{\pt \mu(t,\te) + \int_0^t \pt g(t-u,\te)c(u,\tn)du}{\mu(t,\te) + \int_0^t g(t-u,\te)c(u,\tn)du} c(t,\tn) dt \\
        & - \int_0^1 \left(\pt \mu(t,\te) + \int_0^t \pt g(t-u,\te)c(u,\tn)du\right)dt
    \end{split}
\end{equation*}
and
\begin{equation*}
    \begin{split}
        s_0(\te) = & \int_0^1 \frac{\pt \mu(t,\te) + \int_0^t \pt g(t-u,\te)c(u,\tz)du}{\mu(t,\te) + \int_0^t g(t-u,\te)c(u,\tz)du} c(t,\tz) dt \\
        & - \int_0^1 \left(\pt \mu(t,\te) + \int_0^t \pt g(t-u,\te)c(u,\tz)du\right)dt.
    \end{split}
\end{equation*}
Recalling that $S_n(\hat \te_n) = 0$ according to (\ref{propSn}), using the afore-mentioned Theorem 5.9 boils down to showing

\begin{itemize}
    \item[(i)] $\inf_{\{\te:d(\te, \tz) \geq \ve\}} ||s_0(\te)|| > 0 = s_0(\te_0) \text{ for all } \ve > 0,$
    \item[(ii)] $\sup_{\te \in \Theta} ||\ani S_n(\te) - s_0(\te)|| \pn 0$.
\end{itemize}

While (i) can be proved exactly as in \cite{chehal2013}, we will give a proof for (ii) in two steps. First, we will show \(\sup_{\te \in \Theta} \left| \left| s_n(\te) - s_0(\te) \right| \right| \longrightarrow 0\), for which we write
\begin{equation}
    \begin{split}
        &\left| \left|s_n(\theta) - s_0(\theta)\right| \right| \le  \int_0^1 \int_0^t \left|\left|\pt g(t-u,\te) \right|\right|\left| c(u,\tn) - c(u,\tz) \right| du dt \\
        &~~~~~~~~~~~  +\int_0^1 \left|\left|\frac{\partial_{\theta} \mu(t,\theta) + \int_0^t \partial_{\theta} g(t-u, \theta) c(u,\theta_n) du}{\mu(t,\theta) + \int_0^t g(t-u, \theta) c(u,\theta_n) du} c(t,\theta_n) \right. \right. \\
        &~~~~~~~~~~~- \left. \left. \frac{\partial_{\theta} \mu(t,\theta) + \int_0^t \partial_{\theta} g(t-u, \theta) c(u,\theta_0) du}{\mu(t,\theta) + \int_0^t g(t-u, \theta) c(u,\theta_0) du} c(t,\theta_0) \right| \right|  dt. \label{term1}
    \end{split}
\end{equation}
The first term on the right hand side is dealt with using Lemma \ref{hconv} and uniform boundedness of $\pt g$, while the second term is split up into
\begin{equation*}
    \begin{split}
        & \int_0^1 \left| \left| \frac{\partial_{\theta} \mu(t,\theta) + \int_0^t \partial_{\theta} g(t-u, \theta) c(u,\theta_n) du}{\mu(t,\theta) + \int_0^t g(t-u, \theta) c(u,\theta_n) du} - \frac{\partial_{\theta} \mu(t,\theta) + \int_0^t \partial_{\theta} g(t-u, \theta) c(u,\theta_0) du}{\mu(t,\theta) + \int_0^t g(t-u, \theta) c(u,\theta_0) du} \right| \right| c(t,\theta_n) dt \\
        & ~~~~~~~~+ \int_0^1  \left| \left|\frac{\partial_{\theta} \mu(t,\theta) + \int_0^t \partial_{\theta} g(t-u, \theta) c(u,\theta_0) du}{\mu(t,\theta) + \int_0^t g(t-u, \theta) c(u,\theta_0) du} \right| \right| \left|c(t,\theta_n) - c(t,\theta_0) \right| dt.
\end{split}
\end{equation*}
To bound the first summand we use Corollary \ref{gconv} as well as the uniform boundedness from below of $\mu$ in order to deduce that the difference of the fractions converges to zero uniformly in $\te \in \Te$. Uniform convergence to zero of the entire term then follows from uniform boundedness of $c$. The second summand is treated similarly to the first term on the right hand side of (\ref{term1}).

To finish the proof of (ii) we have to show \(\sup_{\te \in \Theta} \left| \left| a_n^{-1}S_n(\te) - s_n(\te) \right| \right| \pn 0\). We will start with proving the convergence for a fixed \(\te \in \Theta\). Clearly,
\begin{equation*}
    \begin{split}
        &\left| \left| \ani S_n(\te) - s_n(\te)\right|\right| \\
        & ~~\le \left| \left| \int_0^1 \frac{\pt \mu(t,\te) + \int_0^t \pt g(t-u,\te)d\ani N_n(u)}{\mu(t,\te) + \int_0^t g(t-u,\te)d \ani N_n(u)} d \ani N_n(t) \right. \right.\\
        &~~~~~~~~- \left. \left. \int_0^1 \frac{\pt \mu(t,\te) + \int_0^t \pt g(t-u,\te)c(u,\tn)du}{\mu(t,\te) + \int_0^t g(t-u,\te)c(u,\tn)du} c(t,\tn) dt \right| \right| \\
        & ~~~~+  \int_0^1 \left| \left| \int_0^t \pt g(t-u,\te)c(u,\tn)du - \int_0^t \pt g(t-u,\te)d \ani N_n(u)\right| \right|  dt.
    \end{split}
\end{equation*}
All summands converge to zero in view of a repeated application of Corollary \ref{lambdaconv}. 
%
%
%
To prove that this convergence is uniform we can follow the same steps as in \cite{chehal2013}. The only difference is that we additionally need to show that for every $\eps> 0$ there exists $\de > 0$ such that $\| \te_1 - \te_2 \| < \de$ implies $\|s_n(\te_1) - s_n(\te_2)\| < \ve$. This property, however, follows easily from $\sup_{\te \in \Theta} \left| \left| s_n(\te) - s_0(\te) \right| \right| \longrightarrow 0$ and the fact that such a claim already holds for $s_0$.

To finish the proof of Lemma \ref{lem1} we recall (\ref{propSn}). It is clear that Lemma \ref{lem1}(b) follows once \(\anis S_n(\tn) \tol N(0, \mI(\tz))\) and Lemma \ref{lem1}(c) have been established. For the latter, note that \(\sup_{\te \in \Theta} \left| \left| \ani \pt S_n(\te) - \pt s_0(\te) \right| \right| \pn 0\) can be established in the same way as (ii) above. \(\pt s_0(\tz) = -\mI(\tz)\) and \(\Wte_n \conv \tz\) together with continuity of $\pt s_0$ then give the claim. Finally, the other claim follows as in Theorem 2 of \cite{chehal2013}, replacing \(\tz\) by \(\tn\) up until the definition of \(\langle \tilde{M} \rangle_n (t)\). Then \(\tilde{M}_n(t)\) converges to \(v(t)\) without changing the definition of the latter.

\subsection{Proof of Theorem \ref{thm1}} \label{subsect:pth1}
As a first step in the proof, we will give a preliminary comment from which several simplifications follow. Some of these are direct, others need more attention and will be treated in several lemmas.

Suppose that a decomposition of the form
\begin{equation} \label{simpldec}
    \sqrt{a_n} G_n(u) = \ga_n F_n(u) + H_n(u)
\end{equation}
with $\ga_n \pn 1$ and $\vert \vert H_n \vert \vert \pn 0$ holds. Then, using \(a_n \vert \vert G_n \vert \vert^2 = \ga_n^2 \vert \vert F_n \vert \vert^2 +\vert \vert H_n \vert \vert^2 + 2 \ga_n \langle F_n, H_n \rangle\), as well as $\ga_n^2 \pn 1$, $\vert \vert H_n \vert \vert^2 \pn 0$ and \(\langle F_n, H_n \rangle \le \vert \vert F_n \vert \vert \ \vert \vert H_n \vert \vert\) plus Slutsky's lemma and the continuous mapping theorem, it is obviously enough to prove weak convergence of $\vert \vert F_n \vert \vert$ towards a limiting distribution with the desired properties in order to deduce the claim for $\sqrt{a_n} \vert \vert G_n \vert \vert$.

We use (\ref{simpldec}) for two quick simplifications before we need to get more sophisticated. On one hand, Lemma \ref{lem1}(a) allows to replace
\begin{equation*}
    \sqrt{a_n} G_n(u) = \sqrt{a_n} \left(L_n(u,\hat \te_n) - L(u) \right) = \frac{\sqrt{a_n}}{N_n(1)} \sum_{i \ge 1} \left(k_n(u, \hat \te_n, t_{i-1}, t_i) - L(u) \right) 1_{\{t_i \le 1\}}
\end{equation*}
with \(\left(\sqrt{a_n}/\ka_n\right) \sum_{i \ge 1} \left(k_n(u, \hat \te_n, t_{i-1}, t_i) - L(u) \right) 1_{\{t_i \le 1\}}\). Second, using the uniform boundedness of $g_n$ and $L$ as well as integrability of $\be$ and $a_n \to \infty$ we can equivalently discuss
\begin{equation} \label{simpl1}
    \frac{\sqrt{a_n}}{\ka_n} \sum_{i \ge 1} \left(k_n(u, \hat \te_n, t_{i-1}, t_i) - L(u) \right) 1_{\{t_{i-1} \le 1\}}
\end{equation}
instead, which later allows for a usage of martingale methods.

We will now state several decompositions of the random function in (\ref{simpl1}), and to keep the notation readable we will provide these in the case where $\Te$ is one-dimensional. The extension to an arbitrary dimension $d$ bears no additional difficulty. We first write
\begin{equation} 
    \begin{split}
        &k_n(u, \hat \te_n, t_{i-1}, t_i) - L(u) \label{dec1} \\ 
        =&\left(k_n(u, \hat \te_n, t_{i-1}, t_i) - k_n(u, \te_n, t_{i-1}, t_i) \right) + \left(k_n(u, \te_n, t_{i-1}, t_i) - L(u) \right), 
    \end{split}
\end{equation}
where the first summand represents the bias due to the estimation of $\te_n$ while the second one can essentially be interpreted as the usual randomness of having to estimate a Laplace transform from i.i.d.\ observations. To provide an intuition for the latter interpretation, setting \(\fF_t = \fF_t^n = \si \left(N_n(s) ~\middle|~ s \le t \right)\) we have that the conditional distribution of \(\int_{t_{i-1}}^{t_i} a_n \mu_n(t,\te_n) dt = \La_n(t_i,\te_n) - \La_n(t_{i-1},\te_n)\) given $\fF_{t_{i-1}}$ is the standard exponential.

We will now provide a series of simplifications for the first summand in (\ref{dec1}) which we can decompose as
\begin{equation} \label{dec2}
    \begin{split}
        & k_n(u, \hat \te_n, t_{i-1}, t_i) - k_n(u, \te_n, t_{i-1}, t_i) = \int_{\te_n}^{\hat \te_n} q_n(u,r,t_{i-1}, t_i) dr\\
        & ~~= \left(\hat \te_n - \te_n\right) q_n(u,\te_n,t_{i-1}, t_i) + \int_{\te_n}^{\hat \te_n} \int_{\te_n}^r  \Wq_n(u,z,t_{i-1}, t_i) dz dr \\
        & ~~= \left(\hat \te_n - \te_n\right) \E\left(q_n(u,\te_n,t_{i-1}, t_i) \middle | \fF_{t_{i-1}} \right) \\
        & ~~~~~~~~+ \left(\hat \te_n - \te_n\right) \left \{q_n(u,\te_n,t_{i-1}, t_i) -\E\left(q_n(u,\te_n,t_{i-1}, t_i) \middle | \fF_{t_{i-1}} \right) \right\} \\
        & ~~~~~~~~+ \int_{\te_n}^{\hat \te_n} \int_{\te_n}^r \Wq_n(u,z,t_{i-1}, t_i) dz dr
    \end{split}
\end{equation}
where we have set
\begin{equation*}
    q_n(u,\te,r,s) = \frac{\partial}{\partial \te} k_n(u,\te,r,s) \quad \text{and} \quad \Wq_n(u,\te,r,s) = \frac{\partial^2}{\partial \te^2} k_n(u,\te,r,s).
\end{equation*}
Using (\ref{simpldec}), the next lemma proves that the latter two terms in (\ref{dec2}) do not contribute in the asymptotics. It proof, as well for several other results, will be given later. 

\begin{lemma} \label{lem2}
    Let
    \begin{equation*}
        \begin{split}
            \overline G_n(u) = & \frac{\sqrt{a_n}}{\ka_n} \left(\hat \te_n - \te_n\right) \sum_{i \ge 1} \left \{q_n(u,\te_n,t_{i-1}, t_i) -\E\left(q_n(u,\te_n,t_{i-1}, t_i) \middle | \fF_{t_{i-1}} \right) \right\} 1_{\{t_{i-1} \le 1\}} \\
            & ~~~~+ \frac{\sqrt{a_n}}{\ka_n}  \sum_{i \ge 1} \int_{\te_n}^{\hat \te_n} \int_{\te_n}^r \Wq_n(u,z,t_{i-1}, t_i) dz dr  1_{\{t_{i-1} \le 1\}} \\
            =: & \sqrt{a_n} \left(\hat \te_n - \te_n\right)\overline G_n^{(1)}(u) + \overline G_n^{(2)}(u).
        \end{split}
    \end{equation*}
    Then under Condition \ref{condmain}, \(\vert \vert \overline G_n \vert \vert^2 \pn 0\).
\end{lemma}

Following \eqref{simpl1}, \eqref{dec1} and Lemma \ref{lem2} we thus only have to establish the asymptotics associated with the random function
\begin{equation} \label{simpl2}
    \begin{split}
        &\frac{\sqrt{a_n}}{\ka_n} \sum_{i \ge 1} \left\{ \left(\hat \te_n - \te_n\right) \E\left(q_n(u,\te_n,t_{i-1}, t_i) \middle | \fF_{t_{i-1}} \right) \right. \\
        & ~~~~~~~~+ \left. \left(k_n(u, \te_n, t_{i-1}, t_i) - L(u) \right)\right\} 1_{\{t_{i-1} \le 1\}}.
    \end{split}
\end{equation}
We provide two further auxiliary results regarding the first summand above.

\begin{lemma} \label{lem3}
    Set
    \begin{equation*}
        h_n(u,\te_n,t_{i-1}) = \E\left(q_n(u,\te_n,t_{i-1}, t_i) \middle | \fF_{t_{i-1}} \right) \quad \text{and} \quad h(u,\te_0,t) = -\frac{u}{(u+1)^2} \frac{\frac{\partial}{\partial \te} \mu(t,\te_0)}{C(1,\te_0)}
    \end{equation*}
    as well as
    \begin{equation*}
        \begin{split}
            \widehat G_n(u) = & a_n^{1/2} \left(\hat \te_n - \te_n\right)  \left(\frac{1}{\ka_n} \sum_{i \ge 1} h_n(u,\te_n,t_{i-1})  1_{\{t_{i-1} \le 1\}} - \int_0^1 h(u,\te_0,t) dt  \right) \\
            =: & a_n^{1/2} \left(\hat \te_n - \te_n\right) \widehat G_n^{(1)}(u).
        \end{split}
    \end{equation*}
    Then under Condition \ref{condmain}, \(\vert \vert \widehat G_n \vert \vert^2 \pn 0\).
\end{lemma}

For the other result we utilize the representation
\begin{equation*}
    0 = a_n^{-1/2} S_n(\hat \te_n) = a_n^{-1/2} S_n(\te_n) + a_n^{-1} \frac{\partial}{\partial \te} S_n(\Wte_n) a_n^{1/2} \left(\hat \te_n - \te_n\right)
\end{equation*}
which was introduced prior the statement of Lemma \ref{lem1}. Using part (c) of that lemma, we can, up to a null set, write
\begin{equation*}
    a_n^{1/2} \left(\hat \te_n - \te_n\right) = \frac{a_n^{-1/2} S_n(\te_n) }{- a_n^{-1} \frac{\partial}{\partial \te} S_n(\Wte_n)}.
\end{equation*}
According to (\ref{simpldec}), Lemma \ref{lem1}(c), Lemma \ref{lem3} and the previous simplifications leading to (\ref{simpl2}) it remains to establish weak convergence of $\vert \vert \widetilde G_n \vert \vert$ with
\begin{equation} \label{deftil}
    \begin{split}
        \widetilde G_n(u) = & a_n^{-1/2} \left(\frac{S_n(\te_n) }{\iI(\te_0)} \int_0^1 h(u,\te_0,t) dt \right. \\
        & ~~~~~~~~ \left. + \frac{1}{C(1,\te_0)} \sum_{i \ge 1}\left(k_n(u, \te_n, t_{i-1}, t_i) - L(u) \right) 1_{\{t_{i-1} \le 1\}} \right).
    \end{split}
\end{equation}
In fact, we will only prove that the $\hH$-valued random variable $\widetilde G_n(\cdot)$ converges weakly to an $\hH$-valued Gaussian limiting variable $G_{\te_0}(\cdot)$, as the result then follows from an application of the continuous mapping theorem. For the proof of the afore-mentioned weak convergence in a Hilbert space we will rely in particular on Theorem 1.8.4 in \cite{vanwel1996} which contains two sufficient conditions.

The first condition which needs to be checked is asymptotic finite-dimensionality. Denoting an orthonormal basis of $\hH$ by $\{ e_j ~|~ j \in J\}$ this amounts to showing that for every $\eps, \de > 0$ there exists a finite set $I \subset J$ such that \(\limsup_{n \to \infty} \P\left( \sum_{j \notin I} \langle \WG_n, e_j \rangle^2 > \de \right) < \eps\).

Clearly it is sufficient to prove this property for the two summands in \eqref{deftil} separately. For the first summand it is relatively simple. First, utilizing Lemma \ref{lem1}(c) and (\ref{simpldec}) we may equivalently discuss \(a_n^{1/2} \left(\hat \te_n - \te_n\right) \int_0^1 h(u,\te_0,t) dt\). 
Now, Lemma \ref{lem1}(b) implies tightness of the sequence $\sqrt{a_n} \left(\hat \te_n - \te_n\right)$, i.e.\ the existence of $A > 0$ such that \(\limsup_{n \to \infty} \P\left(\left|\sqrt{a_n} \left(\hat \te_n - \te_n\right) \right| > A \right) \le \eps/2\). 
Hence only the existence of $I$ such that
\begin{equation*}
    \P\left(\sum_{j \notin I} \left \langle \int_0^1 h(u,\te_0,t) dt, e_j \right \rangle^2  > \frac{\de}{A^2} \right) \le \frac{\eps}2
\end{equation*}
needs to be shown, and this is an easy consequence of Bessel's inequality since Condition \ref{condmain}(c) implies boundedness of $\vert \vert \int_0^1 h(u,\te_0,t) dt \vert \vert^2$.

The discussion of the other summand in \eqref{deftil} is more subtle. Note first that \(\E\left(k_n(u, \te_n, t_{i-1}, t_{i}) - L(u) \middle| \fF_{t_{i-1}} \right) =0\) holds for every $i$, utilizing that the $\fF_{t_{i}-1}$-conditional distribution of $k_n(u, \te_n, t_{i-1}, t_{i})$ equals the one of $\exp(-uX)$ for $X \sim \exp(1)$ and is independent of $\fF_{t_{i}-1}$. Hence, by Fubini's theorem we obtain
\begin{equation} \label{fub}
    \begin{split}
        & \E \left( \left \langle k_n(u, \te_n, t_{i-1}, t_i) - L(\cdot), e_j \right \rangle \middle| \fF_{t_{i-1}} \right) \\
        & ~~= \int_0^\infty \E\left(k_n(u, \te_n, t_{i-1}, t_i) - L(u) \middle| \fF_{t_{i-1}} \right) e_j(u) \be(u) du =0,
    \end{split}
\end{equation}
and for any choice of $i_1 \neq i_2$ we thus have
\begin{equation*}
    \E\left(\left \langle k_n(\cdot, \te_n, t_{i_1-1}, t_{i_1}) - L(\cdot), e_j \right \rangle 1_{\{t_{{i_1}-1} \le 1\}} \left \langle k_n(\cdot, \te_n, t_{i_2-1}, t_{i_2}) - L(\cdot), e_j \right \rangle 1_{\{t_{{i_2}-1} \le 1\}} \right) = 0
\end{equation*}
by successive conditioning. To summarize,
\begin{equation*}
    \begin{split}
        &\E \left(\left \langle \frac{a_n^{-1/2}}{C(1,\te_0)} \sum_{i \ge 1}\left(k_n(u, \te_n, t_{i-1}, t_i) - L(u) \right) 1_{\{t_{i-1} \le 1\}}, e_j \right \rangle^2  \right) \\
        & ~~= \frac{a_n^{-1}}{C^2(1,\te_0)} \sum_{i \ge 1} \E \left(\left \langle k_n(u, \te_n, t_{i-1}, t_i) - L(u), e_j \right \rangle^2 1_{\{t_{i-1} \le 1\}} \right) \\
        & ~~= \frac{a_n^{-1}}{C^2(1,\te_0)} \sum_{i \ge 1} \E \left( \E\left(\left \langle k_n(u, \te_n, t_{i-1}, t_i) - L(u), e_j \right \rangle^2 \middle| \fF_{t_{i-1}} \right) 1_{\{t_{i-1} \le 1\}} \right) \\
        & ~~= \frac{a_n^{-1}}{C^2(1,\te_0)} \E\left(\left \langle \exp(-uX) - L(u), e_j(u) \right \rangle^2 \right) \E(N_n(1)+1) \\ \le& K \E\left(\left \langle \exp(-uX) - L(u) , e_j(u) \right \rangle^2 \right)
    \end{split}
\end{equation*}
where we have used Lemma \ref{lem1}(a) and Condition \ref{condmain}(c) in the last step. The claim then follows again from Bessel's inequality, this time in connection with dominated convergence and Markov's inequality, utilizing
\begin{equation*}
    \sum_{j \in J} \E\left(\left \langle \exp(-uX) - L(u), e_j(u) \right \rangle^2 \right) \le \E \left(\left \vert \left \vert \exp(-uX) - L(u) \right \vert \right \vert^2 \right) < \infty.
\end{equation*}

The second condition of Theorem 1.8.4 in \cite{vanwel1996} regards weak convergence of
\begin{equation} \label{Gnf}
    \begin{split}
        \left \langle \widetilde G_n, f \right \rangle = & \frac{a_n^{-1/2} S_n(\te_n) }{\iI(\te_0)} \left \langle\int_0^1 h(u,\te_0,t) dt, f(u) \right \rangle \\
        & ~~~~+ \frac{a_n^{-1/2}}{C(1,\te_0)} \sum_{i \ge 1}\left \langle k_n(u, \te_n, t_{i-1}, t_i) - L(u), f(u) \right \rangle 1_{\{t_{i-1} \le 1\}}.
    \end{split}
\end{equation}
towards $\langle G_{\te_0}, f \rangle$ for each $f \in \hH$. Here, $G_{\te_0}$ is an $\hH$-valued Gaussian random variable. Precisely, we will show the afore-mentioned weak convergence with \(\langle G_{\te_0}, f \rangle \sim \nN\left(0, \si^2(f) \right)\) for \(f \in \hH\), which is enough to deduce Gaussianity of $G_{\te_0}$ and determines its distribution uniquely. Here,
\begin{equation*}
    \begin{split}
        \si^2(f) = & \int_0^1 \left( \left(\frac{\langle\int_0^1 h(u,\te_0,t) dt, f(u)  \rangle}{\iI(\te_0)}\right)^2 \frac{ \left(\frac{\partial}{\partial \te}\mu(s,\te_0) \right)^2}{\mu(s,\te_0)} \right. \\
        & ~~~~+ \left. 2 \frac{\langle\int_0^1 h(u,\te_0,t) dt, f(u)  \rangle}{\iI(\te_0)C(1,\te_0)} \Phi(f) \frac{\partial}{\partial \te} \mu(s,\te_0) + \frac{\De(f)}{C(1,\te_0)^2} \mu(s,\te_0)\right) ds
    \end{split}
\end{equation*}
with \(\Phi(f) =  \int_0^\infty \frac{u}{(1+u)^2} f(u) \be(u)  du\) and
\begin{equation*}
    \De(f) = \int_0^\infty \int_0^\infty \left(\frac 1{1+u+v} -  \frac 1{(1+u)(1+v)} \right) \be(u) f(u) \be(v) f(v) du dv.
\end{equation*}
Hence, let $f \in \hH$ be fixed from now on.

For the proof of the weak convergence stated above we will rely on a martingale central limit theorem provided as Theorem 2.2.13 in \cite{jacpro2012} for which we need some preparation. Loosely speaking, the problem is that both summands on the right hand side of (\ref{Gnf}) can be regarded as end points of martingales but with respect to different filtrations (the first one with respect to $\fF_t$ for every $t \in [0,1]$, the second one with respect to $\fF_{t_{i-1}}$ for every $i$). Hence we are looking for an approximation of $\langle G_{\te_0}, f \rangle$ by an end point of a martingale with respect to the same filtration. For this, we use a ``big blocks, small blocks''-strategy. Let $p \in \N$ and $0 < \varrho < 1/2$ be fixed, and set $\ell_n = \lfloor{a_n^\varrho} \rfloor$, as well as $c_{n,p}(j) = j(p+1)\ell_n$ and $d_{n,p}(j) = j(p+1)\ell_n + p \ell_n$, and recall that one way to prove the weak convergence $Y_n \tol Z$ (for generic random variables) is to show
\begin{align}
&Z_{n,p} \tol Z_p \text{ as } n \to \infty, \text{ for any fixed } p \in \N, \label{aux1} \\
&Z_p \tol Z \text{ as } p \to \infty, \label{aux2}\\
&\lim_{p \to \infty} \limsup_{n \to \infty} \P(|Y_n - Z_{n,p}| \geq \eta) = 0, \text{ for any } \eta > 0, \label{aux3}
\end{align}
for auxiliary random variables $Z_{n,p}$ and $Z_p$. We will begin with a version of (\ref{aux3}) in our setting, for which we recall the definition of $U_n(t,\te)$ in (\ref{defUn}).

\begin{lemma} \label{lem5}
    Let
    \begin{equation*}
        G_{n,p}(f) = \sqrt{\frac{(p+1)\ell_n}{a_n}}\sum_{j=0}^{\lfloor \frac{a_n}{(p+1)\ell_n} \rfloor - 1}  \left(\frac{\langle\int_0^1 h(u,\te_0,t) dt, f(u)  \rangle}{\iI(\te_0)} \xi_{p,j}^n + \frac 1{C(1,\te_0)} \ze_{p,j}^n(f)\right)
    \end{equation*}
    with
    \begin{equation*}
        \begin{split}
            \xi_{p,j}^n & = \frac{1}{\sqrt{(p+1)\ell_n}} \sum_{i \ge 1} \left(U_n(t_{i-1},\te_n) 1_{\{\frac{c_{n,p}(j)}{a_n} < t_{i-1} \le \frac{d_{n,p}(j)}{a_n}\}} \right. \\
            &  ~~~~~~~~~~~~~~~~- \left. \E\left(U_n(t_{i-1},\te_n) 1_{\{\frac{c_{n,p}(j)}{a_n} < t_{i-1} \le \frac{d_{n,p}(j)}{a_n}\}} \middle| \fF_{\frac{c_{n,p}(j)}{a_n}} \right) \right)
        \end{split}
    \end{equation*}
    and
    \begin{equation*}
        \begin{split}
            &\ze_{p,j}^n(f)= \frac{1}{\sqrt{(p+1)\ell_n}} \times \\
            & ~~~~~ \sum_{i \ge 1} \left(\left \langle k_n(u, \te_n, t_{i-1}, t_i) - L(u), f(u) \right \rangle 1_{\{\frac{c_{n,p}(j)}{a_n} < t_{i-1} \le \frac{d_{n,p}(j)}{a_n}, t_i \le \frac{c_{n,p}(j+1)}{a_n}\}} \right. \\
            &~~~~~~- \left. \E\left(\left \langle k_n(u, \te_n, t_{i-1}, t_i) - L(u), f(u) \right \rangle 1_{\{\frac{c_{n,p}(j)}{a_n} < t_{i-1} \le \frac{d_{n,p}(j)}{a_n}, t_i \le \frac{c_{n,p}(j+1)}{a_n}\}} \middle| \fF_{\frac{c_{n,p}(j)}{a_n}} \right) \right)
        \end{split}
    \end{equation*}
    Then, for any $\eta > 0$,
    \begin{equation*}
        \lim_{p \to \infty} \limsup_{n \to \infty} \P\left(\left|\left \langle \widetilde G_n, f \right \rangle -G_{n,p}(f)\right| \geq \eta\right) = 0.
    \end{equation*}
\end{lemma}

    By definition, $G_{n,p}(f)$ for each fixed $p$ is the end point of a martingale with respect to the filtration associated with $\fF_{\frac{c_{n,p}(j)}{a_n}}$ for $j=0, 1, \ldots$ Note that each $\ze_{p,j}^n(f)$ is measurable with respect to $\fF_{\frac{c_{n,p}(j+1)}{a_n}}$, for which the additional condition of $t_i \le \frac{c_{n,p}(j+1)}{a_n}$ is needed. However, as seen from (\ref{zeWze}) below, adding this condition does not alter the overall asymptotic behaviour.

    To finish the proof of Theorem \ref{thm1} we need to prove (\ref{aux1}) and (\ref{aux2}) for $G_{n,p}(f)$. In fact, we will only prove the analogue of (\ref{aux1}) in the following result as it is trivial to deduce (\ref{aux2}) afterwards.

    \begin{theorem} \label{thm2}
        Under Condition \ref{condmain}, and for any fixed $f$ and $p$, we have \(G_{n,p}(f) \tol \nN\left(0, \frac p{p+1} \si^2(f)\right)\).
    \end{theorem}

    \subsection{Proof of Lemma \ref{lem2}}
    Regarding the first summand, note that the sequence $\sqrt{a_n} (\hat \te_n - \te_n)$ is tight according to Lemma \ref{lem1}(b). Hence it is sufficient to show
    \begin{equation*}
        \E\left(\left \vert \left \vert \overline G_n^{(1)} \right \vert \right \vert^2 \right) = \int_0^\infty \E\left(\left( G_n^{(1)}(u) \right)^2 \right) \be(u) du =o(1),
    \end{equation*}
    where we have used Fubini's theorem. By construction,
    \begin{equation*}
            \E\left(\left( G_n^{(1)}(u) \right)^2 \right) = \frac 1{\ka_n^2} \sum_{i \ge 1} \E\left( \left (q_n(u,\te_n,t_{i-1}, t_i) -\E\left(q_n(u,\te_n,t_{i-1}, t_i) \middle | \fF_{t_{i-1}} \right) \right)^2 1_{\{t_{i-1} \le 1\}} \right). 
    \end{equation*}
		
    To finish the proof for the first summand we will show uniform boundedness of $q_n(u,\te_n,t_{i-1}, t_i)$, as then
    \begin{equation*}
        \E\left(\left( G_n^{(1)}(u) \right)^2 \right) \le K \frac 1{\ka_n^2} \sum_{i \ge 1} \E\left( 1_{\{t_{i-1} \le 1\}} \right) \le K \frac{\E(N_n(1)+1)}{\ka_n^2} = o(1)
    \end{equation*}
    (use Lemma \ref{lem1}(a) and integrability of $\be$) concludes. To prove the afore-mentioned uniform boundedness we write
    \begin{equation*}
        \left|q_n(u,\te,r,s)\right| \le \exp\left(-u \int_r^s a_n \mu(t,\te) dt \right) u a_n \int_r^s \left|\frac{\partial}{\partial \te} \mu(t,\te)\right| dt,
    \end{equation*}
    where we have used Condition \ref{condmain}(c) to justify the exchange of integral and derivative. The same assumption allows to bound the right hand side above by
    \begin{equation*}
        |q_n(u,\te,r,s)| \le K \exp\left(-u \int_r^s a_n \mu(t,\te) dt \right) u\, a_n \int_r^s \mu(t,\te) dt,
    \end{equation*}
    using the fact that $\mu(t,\te)$ is uniformly bounded from below while $|\frac{\partial}{\partial \te} \mu(t,\te)|$ is uniformly bounded from above. $x \exp(-x) \le 1$ for $x \ge 0$ finishes the proof for the first summand.

    For the second summand, note that a similar argument in combination with the boundedness of $x^2 \exp(-x)$ for $x \ge 0$ yields $|\Wq_n(u,\te,r,s)| \le K$. Hence, \(| G_n^{(2)}(u)| \le K \sqrt{a_n} \left|\hat \te_n - \te_n\right|^2 \left(N_n(1)+1\right) / \ka_n \), and Slutsky's lemma in combination with Lemma \ref{lem1}(a) and (b) proves that the right hand side is $o_\P(1)$.
    Since the same upper bound  (up to a possibly different constant) holds for $||G_n^{(2)}||$ as well, the proof is complete. \qed

    \subsection{Proof of Lemma \ref{lem3}}
    Let us start with some remarks regarding integrals with respect to inhomogeneous Poisson processes. For any $\fF$-stopping time $R$, any $s > 0$, and any bounded measurable function $\vp \ge 0$ we have
    \begin{equation*}
        \begin{split}
            \E\left(\sum_{i \ge 1} 1_{\{ R < t_i \le R+s\}} \vp(t_i,\te_n) \middle|\fF_{R}\right) & = \E\left(\int_{R}^{R+s} \vp(t,\te_n) dN_n(t) \middle|\fF_{R}\right) \\
                                                                                                  & = \E\left(a_n \int_{R}^{R+s} \vp(t,\te_n) \mu(t,\te_n) dt \middle|\fF_{R}\right)
        \end{split}
    \end{equation*}
    as well as
    \begin{equation*}
        \Var\left(\int_{R}^{R+s} \vp(t,\te_n) dN_n(t) \middle|\fF_{R}\right) = \E\left(a_n \int_R^{R+s} \vp^2(t,\te_n) \mu(t,\te_n) dt \middle|\fF_{R}\right).
    \end{equation*}

    We now write  $\wG_n^{(1)}(u) = \wG_n^{(2)}(u) + \wG_n^{(3)}(u) + \wG_n^{(4)}(u)$ with
    \begin{align*}
        \wG_n^{(2)}(u) & = \frac{1}{\ka_n} \left(\sum_{i \ge 1} h_n(u,\te_n,t_{i-1})  1_{\{t_{i-1} \le 1\}} - \int_0^1 h_n(u,\te_n,t) a_n \mu(t,\te_n) dt \right), \\
        \wG_n^{(3)}(u) & = \int_0^1 \left( \frac{a_n}{\ka_n} h_n(u,\te_n,t)\mu(t,\te_n) - h(u,\te_n,t)  \right)dt, \\
        \wG_n^{(4)}(u) & = \int_0^1 \left( h(u,\te_n,t) - h(u,\te_0,t)  \right)dt.
    \end{align*}
    Obviously, as in the proof of Lemma \ref{lem2}, it suffices to show \(\E\left(\left \vert \left \vert \wG_n^{(1)} \right \vert \right \vert^2 \right) = o(1)\), and we prove the claim above separately for $\wG_n^{(2)}$, $\wG_n^{(3)}$ and $\wG_n^{(4)}$. The latter is an immediate consequence of Condition \ref{condmain}(c) and dominated convergence.

    For the first one, note that we may write
    \begin{equation*}
        \begin{split}
            \wG_n^{(2)}(u) &= \frac{1}{\ka_n} \left(\sum_{i \ge 1} h_n(u,\te_n,t_{i-1})  1_{\{t_{i-1} \le 1\}} - \int_0^1 h_n(u,\te_n,t) a_n \mu(t,\te_n) dt \right)\\
                           &= \frac 1{\ka_n} \left(\int_0^1 h_n(u,\te_n,t) dN_n(t) - \int_0^1 h_n(u,\te_n,t) a_n \mu(t,\te_n) dt \right) + o_\P(1)
        \end{split}
    \end{equation*}
    where the error term is due to up to two additional summands according to $1_{\{t_{i-1} \le 1\}}$ instead of $1_{\{0 < t_{i} \le 1\}}$. In particular, this error term is uniformly bounded by definition of $h_n$ and $g_n$, and it is then easy to obtain
    \begin{equation*}
        \begin{split}
            & \E\left(\left( \wG_n^{(2)}(u) \right)^2 \right) \\
            & ~~= \E \left( \frac{1}{\ka^2_n} \left(\int_0^1 h_n(u,\te_n,t) dN_n(t) - \int_0^1 h_n(u,\te_n,t) a_n \mu(t,\te_n) dt \right) ^2 \right) + o(1) \\
            & ~~= \frac 1{\ka_n^2}  \int_0^1 h^2_n(u,\te_n,t) a_n \mu(t,\te_n) dt + o(1) = o(1)
        \end{split}
    \end{equation*}
    with again uniformly bounded error terms. Dominated convergence concludes.

    The most involved part of the proof regards $\wG_n^{(3)}$, and we begin with deriving a suitable functional form for $h_n(u,\te_n,t_{i-1})$. On the set where $s > t_{i-1}$, we have
    \begin{equation*}
        \begin{split}
            \P\left(t_{i} \ge s \middle| \fF_{t_{i-1}} \right) & = \P\left(\La(t_{i},\te_n) - \La(t_{i-1},\te_n) \ge \La(s,\te_n) - \La(t_{i-1},\te_n)  \middle| \fF_{t_{i-1}} \right) \\
                                                               & = \exp \left(- \int_{t_{i-1}}^s a_n \mu(r,\te_n) dr \right),
        \end{split}
    \end{equation*}
    which gives the conditional density
    \begin{equation} \label{density}
        f_{t_{i-1}}(s) = \exp \left(- \int_{t_{i-1}}^s a_n \mu(r,\te_n) dr \right) a_n \mu(s,\te_n) 1_{\{s > t_{i-1}\}}.
    \end{equation}
    As a consequence, we have by definition of $q_n(u,\te,r,s)$ and using integration by parts as well as a change of variables
    \begin{equation*}
        \begin{split}
            & h_n(u,\te_n,t_{i-1}) = \int_{t_{i-1}}^{\infty} \exp\left( -u\int_{t_{i-1}}^s a_n \mu(r,\te_n) dr \right) \left(-u a_n \int_{t_{i-1}}^s \frac{\partial}{\partial \te} \mu(r,\te_n) dr \right) \\
            & ~~~~~~~~~~~~~~~~~~~~~~~~ \exp\left( -\int_{t_{i-1}}^s a_n \mu(r,\te_n) dr  \right)a_n \mu(s,\te_n) ds \\
            & ~~= \int_{t_{i-1}}^{\infty}  \exp\left( -(u+1)\int_{t_{i-1}}^s a_n \mu(r,\te_n) dr  \right) \left(-u a_n \int_{t_{i-1}}^s \frac{\partial}{\partial \te} \mu(r,\te_n) dr \right)a_n \mu(s,\te_n) ds \\
            & ~~= -\frac{u}{1+u}\int_{0}^{\infty}  \exp\left( -(u+1)\int_{t_{i-1}}^{t_{i-1}+s} a_n \mu(r,\te_n) dr  \right) a_n \frac{\partial}{\partial \te} \mu(t_{i-1}+s,\te_n) ds.
        \end{split}
    \end{equation*}
    Again, Condition \ref{condmain}(c) allows to exchange integrals and derivatives. Using the shorthand notation \(g_{t,n}(s) = \int_0^s \mu(r+ t, \te_n) dr\) and \(f_{t,n}(s) = \frac{\partial}{\partial \te} \mu(t+s,\te_n)\), and another change of variables, we can hence write
    \begin{equation*}
        \begin{split}
            & h_n(u,\te_n,t)  = -\frac{u}{1+u}\int_{0}^{\infty}  \exp\left( -(u+1) a_n g_{t,n}(s) \right) a_n f_{t,n}(s) ds \\
            & ~~= -f_{t,n}(0)\frac{u}{1+u}\int_{0}^{\infty}   \exp\left( - (u+1)s \frac{g_{t,n}(\frac{s}{a_n}) - g_{t,n}(0)}{\frac s{a_n}} \right) ds \\
            & ~~~~-\frac{u}{1+u}\int_{0}^{\infty}  \exp\left( - (u+1)s \frac{g_{t,n}(\frac{s}{a_n}) - g_{t,n}(0)}{\frac s{a_n}} \right) \left(f_{t,n}\left(\frac{s}{a_n}\right) - f_{t,n}(0) \right) ds \\
            & ~~=: h_n^{(1)}(u,\te_n,t) + h_n^{(2)}(u,\te_n,t).
        \end{split}
    \end{equation*}
    Proving \(\int_0^\infty \left( \int_0^1 \frac{a_n}{\ka_n} h_n^{(2)}(u,\te_n,t)\mu(t,\te_n)dt \right)^2 \be(u) du = o(1)\) follows quickly, as Condition \ref{condmain}(c) grants boundedness of all coefficients as well as
    \begin{equation*}
        \begin{split}
            & \int_{0}^{\infty}  \exp\left( - (u+1)s \frac{g_{t,n}(\frac{s}{a_n}) - g_{t,n}(0)}{\frac s{a_n}} \right) \left(f_{t,n}\left(\frac{s}{a_n}\right) - f_{t,n}(0) \right) ds  \\
            & ~~~~\le K \int_{0}^{\infty} \frac{s}{a_n} \exp\left(-Ks \right) ds \le \frac{K}{a_n}.
        \end{split}
    \end{equation*}

    We finally prove
    \begin{equation} \label{stepfin}
        \int_0^\infty \left(\int_0^1 \left( \frac{a_n}{\ka_n}  h_n^{(1)}(u,\te_n,t) \mu(t,\te_n) - h(u,\te_n,t)  \right)dt \right)^2 \be(u) du = o(1).
    \end{equation}
    Note that we can rewrite
    \begin{equation*}
        \begin{split}
            h(u,\te_n,t) & = -\frac u{u+1} \frac 1{C(1,\te_0)} \frac{\partial}{\partial \te} \mu(t,\te_n) \frac{1}{(u+1)\mu(t,\te_n)}  \mu(t, \te_n) \\
                         & = -\frac u{u+1} \frac{a_n}{\ka_n} f_{t,n}(0)  \frac{1}{(u+1)g'_{t,n}(0)}  \mu(t, \te_n),
        \end{split}
    \end{equation*}
    utilizing Condition \ref{condmain}(c). Hence, using dominated convergence three times, (\ref{stepfin}) follows from showing
    \begin{equation*}
        \lim_{n \to \infty }\left\{ \exp\left( - (u+1)s \frac{g_{t,n}(\frac{s}{a_n}) - g_{t,n}(0)}{\frac s{a_n}} \right)  - \exp\left( - (u+1)s g'_{t,n}(0)\right)  \right\} = 0,
    \end{equation*}
    uniformly in $u$, $t$ and $s$. The latter is an easy consequence from Lipschitz continuity of $x \mapsto \exp(-x)$ on $[-K, \infty)$ and Condition \ref{condmain}(c). \qed

    \subsection{Proof of Lemma \ref{lem5}}
    Recall the identities given at the beginning of the proof of Lemma \ref{lem3}.
    Utilizing
    \begin{equation} \label{intU}
        \sum_{i \ge 1} U_n(t_{i-1},\te_n) 1_{\{R < t_{i-1} \le R+s\}} = \int_{R}^{R+s} U_n(t,\te_n) dN_n(t)
    \end{equation}
    and the definition of $U_n(t,\te)$ in (\ref{defUn}), we can then write
    \begin{equation*}
        S_n(\te_n) = \int_0^1 U_n(t,\te_n) dM_n(t,\te_n) = \sum_{i \ge 1} U_n(t_{i-1},\te_n)  1_{\{0 < t_{i-1} \le 1\}} -a_n  \int_0^1 \frac{\partial}{\partial \te} \mu(t,\te)dt,
    \end{equation*}
    and setting
    \begin{equation} \label{xi1}
        \begin{split}
            & \Bxi_{p,j}^n = \frac{1}{\sqrt{(p+1)\ell_n}} \sum_{i \ge 1} \left(U_n(t_{i-1},\te_n) 1_{\{\frac{d_{n,p}(j)}{a_n} < t_{i-1} \le \frac{c_{n,p}(j+1)}{a_n}\}} \right. \\
            & ~~~~~~~~~~~~~~~~~~~~~~~~~~~~~~~~~- \left. \E\left(U_n(t_{i-1},\te_n) 1_{\{\frac{d_{n,p}(j)}{a_n} < t_{i-1} \le \frac{c_{n,p}(j+1)}{a_n}\}} \middle| \fF_{\frac{d_{n,p}(j)}{a_n}} \right) \right) \\
            & ~~~~= \frac{1}{\sqrt{(p+1)\ell_n}} \left( \int_{\frac{d_{n,p}(j)}{a_n}}^{\frac{c_{n,p}(j+1)}{a_n}} U_n(t,\te_n) dN_n(t) - \E\left(\int_{\frac{d_{n,p}(j)}{a_n}}^{\frac{c_{n,p}(j+1)}{a_n}} U_n(t,\te_n) dN_n(t) \middle|\fF_{\frac{d_{n,p}(j)}{a_n}}\right)\right)
    \end{split}
\end{equation}
we then obtain
\begin{equation*}
    \begin{split}
        & a_n^{-1/2} S_n(\te_n) - \sqrt{\frac{(p+1)\ell_n}{a_n}}\sum_{j=0}^{\lfloor \frac{a_n}{(p+1)\ell_n} \rfloor - 1}  \xi_{p,j}^n \\
        & ~~~~= \sqrt{\frac{(p+1)\ell_n}{a_n}}\sum_{j=0}^{\lfloor \frac{a_n}{(p+1)\ell_n} \rfloor - 1}  \Bxi_{p,j}^n + a_n^{-1/2} \left( \int_{\lfloor \frac{a_n}{(p+1)\ell_n} \rfloor \frac{(p+1)\ell_n}{a_n}}^{1} U_n(t,\te_n) dN_n(t) \right. \\
        & ~~~~~~~~~~~~~~~~- \left. \E\left[ \int_{\lfloor \frac{a_n}{(p+1)\ell_n} \rfloor \frac{(p+1)\ell_n}{a_n}}^{1} U_n(t,\te_n) dN_n(t)  \middle|\fF_{\lfloor \frac{a_n}{(p+1)\ell_n} \rfloor \frac{(p+1)\ell_n}{a_n}}\right]\right).
    \end{split}
\end{equation*} 

Using the final representation for $\Bxi_{p,j}^n$ above we get
\begin{equation*}
    \begin{split}
        & \E \left( \left(\sqrt{\frac{(p+1)\ell_n}{a_n}}\sum_{j=0}^{\lfloor \frac{a_n}{(p+1)\ell_n} \rfloor - 1}  \Bxi_{p,j}^n  \right)^2 \right) \\
        & ~~= \frac{(p+1)\ell_n}{a_n}\sum_{j=0}^{\lfloor \frac{a_n}{(p+1)\ell_n} \rfloor - 1}  \Var\left(\Bxi_{p,j}^n\right) \\
        & ~~= a_n \frac{(p+1)\ell_n}{a_n}\sum_{j=0}^{\lfloor \frac{a_n}{(p+1)\ell_n} \rfloor - 1} \frac 1{(p+1)\ell_n} \int_\frac{d_{n,p}(j)}{a_n}^\frac{c_{n,p}(j+1)}{a_n} \frac{ \left(\frac{\partial}{\partial \te}\mu(t,\te_n) \right)^2}{\mu(t,\te_n)} dt \le \frac{K}{p+1}
    \end{split}
\end{equation*}
where we have applied the martingale structure in the first equality as well as Condition \ref{condmain}(c) and the definitions of $c_{n,p}(j)$ and $d_{n,p}(j)$ for the final bound. Similarly,
\begin{equation*}
    \begin{split}
        &\E\left( \left(a_n^{-1/2} \left( \int_{\lfloor \frac{a_n}{(p+1)\ell_n} \rfloor \frac{(p+1)\ell_n}{a_n}}^{1} U_n(t,\te_n) dN_n(t) \right.\right.\right. \\
        &~~~~~~~~~~~~~~~~- \left.\left.\left. \E\left( \int_{\lfloor \frac{a_n}{(p+1)\ell_n} \rfloor \frac{(p+1)\ell_n}{a_n}}^{1} U_n(t,\te_n) dN_n(t)  \middle|\fF_{\lfloor \frac{a_n}{(p+1)\ell_n} \rfloor \frac{(p+1)\ell_n}{a_n}}\right)\right)\right)^2 \right) \\
        & ~~= \int_{\lfloor \frac{a_n}{(p+1)\ell_n} \rfloor \frac{(p+1)\ell_n}{a_n}}^{1} \frac{ \left(\frac{\partial}{\partial \te}\mu(t,\te_n) \right)^2}{\mu(t,\te_n)} dt \le K \frac{(p+1)\ell_n}{a_n}.
    \end{split}
\end{equation*}
By definition of $\ell_n$ is then easy to deduce
\begin{equation*}
    \lim_{p \to \infty} \limsup_{n \to \infty} \P\left(\left|a_n^{-1/2} S_n(\te_n) - \sqrt{\frac{(p+1)\ell_n}{a_n}}\sum_{j=0}^{\lfloor \frac{a_n}{(p+1)\ell_n} \rfloor - 1}  \xi_{p,j}^n  \right| \geq \eta\right) = 0
\end{equation*}
which (together with the uniform boundedness of the other factors appearing) finishes the first part of the proof.

The second part of the proof is slightly more involved. We need to reproduce most of the arguments given previously on one hand, but we also need to discuss the additional condition of $t_i \le \frac{c_{n,p}(j+1)}{a_n}$. We will start with the former and, setting
\begin{align*}
    \Wze_{p,j}^n(f) &= \frac{1}{\sqrt{(p+1)\ell_n}} \sum_{i \ge 1} \left \langle k_n(u, \te_n, t_{i-1}, t_i) - L(u), f(u) \right \rangle 1_{\{\frac{c_{n,p}(j)}{a_n} < t_{i-1} \le \frac{d_{n,p}(j)}{a_n}\}},\\
    \Bze_{p,j}^n(f) &= \frac{1}{\sqrt{(p+1)\ell_n}} \sum_{i \ge 1} \left \langle k_n(u, \te_n, t_{i-1}, t_i) - L(u), f(u) \right \rangle 1_{\{\frac{d_{n,p}(j)}{a_n} < t_{i-1} \le \frac{c_{n,p}(j+1)}{a_n}\}},
\end{align*}
discuss
\begin{equation*}
    \begin{split}
        &a_n^{-1/2} \sum_{i \ge 1}\left \langle k_n(u, \te_n, t_{i-1}, t_i) - L(u), f(u) \right \rangle 1_{\{t_{i-1} \le 1\}} - \sqrt{\frac{(p+1)\ell_n}{a_n}}\sum_{j=0}^{\lfloor \frac{a_n}{(p+1)\ell_n} \rfloor - 1}  \Wze_{p,j}^n(f) \\
        & ~~= \sqrt{\frac{(p+1)\ell_n}{a_n}}\sum_{j=0}^{\lfloor \frac{a_n}{(p+1)\ell_n} \rfloor - 1}  \Bze_{p,j}^n(f) \\
        & ~~~~~~~~+ a_n^{-1/2}  \sum_{i \ge 1} \left \langle k_n(u, \te_n, t_{i-1}, t_i) - L(u), f(u) \right \rangle 1_{\{\lfloor \frac{a_n}{(p+1)\ell_n} \rfloor \frac{(p+1)\ell_n}{a_n} < t_{i-1} \le 1\}}.
    \end{split}
\end{equation*}
First, applying (\ref{fub}) with $f$ in place of $e_j$ gives
\begin{equation*}
    \begin{split}
        \E \left(\left \langle k_n(u, \te_n, t_{i-1}, t_i) - L(u), f(u) \right \rangle 1_{\{\frac{c_{n,p}(j)}{a_n} < t_{i-1} \le \frac{d_{n,p}(j)}{a_n}\}} \middle| \fF_\frac{c_{n,p}(j)}{a_n} \right) =0
    \end{split}
\end{equation*}
after successive conditioning utilizing
\begin{equation} \label{condhilf}
    \begin{split}
        \E\left(1_{\{ R \le S \}} X \middle| \fF_R \right) = \E\left(1_{\{ R \le S \}} \E\left(X \middle| \fF_S \right) \middle| \fF_R \right)
    \end{split}
\end{equation}
for a generic random variable $X$ and generic stopping times $R$ and $S$. By uniform boundedness of $\left \langle k_n(u, \te_n, t_{i-1}, t_i) - L(u), f(u) \right \rangle$ for a fixed $f$ we then get
\begin{equation*}
    \begin{split}
        & \E \left( \left(\sqrt{\frac{(p+1)\ell_n}{a_n}}\sum_{j=0}^{\lfloor \frac{a_n}{(p+1)\ell_n} \rfloor - 1}  \Bze_{p,j}^n(f)  \right)^2 \right) = \frac{(p+1)\ell_n}{a_n}\sum_{j=0}^{\lfloor \frac{a_n}{(p+1)\ell_n} \rfloor - 1}  \Var\left(\Bze_{p,j}^n(f)\right) \\
        & ~~~~~~~~\le \frac{1}{a_n}\sum_{j=0}^{\lfloor \frac{a_n}{(p+1)\ell_n} \rfloor - 1}  \E\left( \left(N\left(\frac{c_{n,p}(j+1)}{a_n}\right) - N\left(\frac{d_{n,p}(j)}{a_n}\right) \right)\right) \\
        & ~~~~~~~~\le \frac{1}{a_n}\sum_{j=0}^{\lfloor \frac{a_n}{(p+1)\ell_n} \rfloor - 1} a_n \frac{K \ell_n}{a_n} \le \frac{K}{p+1},
    \end{split}
\end{equation*}
and we also have
\begin{equation*}
    a_n^{-1/2}  \sum_{i \ge 1} \E\left( \left|\left \langle k_n(u, \te_n, t_{i-1}, t_i) - L(u), f(u) \right \rangle \right| 1_{\{\lfloor \frac{a_n}{(p+1)\ell_n} \rfloor \frac{(p+1)\ell_n}{a_n} < t_{i-1} \le 1\}} \right) \le K \frac{(p+1)\ell_n}{\sqrt{a_n}}
\end{equation*}
which converges to zero as $n \to \infty$, by definition of $\ell_n$. For these terms we can then conclude as in the first step of the proof, ignoring bounded factors again.

Finally, we need to bound
\begin{equation*}
    \begin{split}
        & \E\left( \left( \sqrt{\frac{(p+1)\ell_n}{a_n}}\sum_{j=0}^{\lfloor \frac{a_n}{(p+1)\ell_n} \rfloor - 1}  \left(\Wze_{p,j}^n(f) - \ze_{p,j}^n(f) \right) \right)^2 \right) \\
        & ~~= \frac{(p+1)\ell_n}{a_n} \sum_{j=0}^{\lfloor \frac{a_n}{(p+1)\ell_n} \rfloor - 1}\E\left(  \left(\Wze_{p,j}^n(f) - \ze_{p,j}^n(f)\right)^2 \right),
    \end{split}
\end{equation*}
where we have used the martingale structure for both summands again. Note that, applying (\ref{fub}) again,
\begin{equation*}
    \begin{split}
        &\Wze_{p,j}^n(f) - \ze_{p,j}^n(f)  =  \frac{1}{\sqrt{(p+1)\ell_n}} \sum_{i \ge 1} \left \langle k_n(u, \te_n, t_{i-1}, t_i) - L(u), f(u) \right \rangle \times \\
        & ~~~\left(1_{\{\frac{c_{n,p}(j)}{a_n} < t_{i-1} \le \frac{d_{n,p}(j)}{a_n}, t_i > \frac{c_{n,p}(j+1)}{a_n} \}} - \E\left(1_{\{\frac{c_{n,p}(j)}{a_n} < t_{i-1} \le \frac{d_{n,p}(j)}{a_n}, t_i > \frac{c_{n,p}(j+1)}{a_n} \}} \middle|  \fF_{\frac{c_{n,p}(j)}{a_n}}\right) \right).
    \end{split}
\end{equation*}
Clearly, for any fixed $j$ the number of indices $i$ satisfying $\frac{c_{n,p}(j)}{a_n} < t_{i-1} \le \frac{d_{n,p}(j)}{a_n}$ and $t_i > \frac{c_{n,p}(j+1)}{a_n}$ is at most one. Thus, by uniform boundedness again,
\begin{equation} \label{zeWze}
    \E\left( \left( \sqrt{\frac{(p+1)\ell_n}{a_n}}\sum_{j=0}^{\lfloor \frac{a_n}{(p+1)\ell_n} \rfloor - 1}  \left(\Wze_{p,j}^n(f) - \ze_{p,j}^n(f) \right) \right)^2 \right) \le \frac{K}{(p+1)\ell_n}
\end{equation}
which finishes the proof. \qed

\subsection{Proof of Theorem \ref{thm2}}
According to Theorem 2.2.13 in \cite{jacpro2012} it is enough to prove
\begin{equation*}
    \sum_{j=0}^{\lfloor \frac{a_n}{(p+1)\ell_n} \rfloor - 1} \left(\frac{(p+1)\ell_n}{a_n}\right)^{2} \E \left( \left|\frac{\langle\int_0^1 h(u,\te_0,t) dt, f(u)  \rangle}{\iI(\te_0)} \xi_{p,j}^n + \frac 1{C(1,\te_0)} \ze_{p,j}^n(f) \right|^4\right) =o(1)
\end{equation*}
as well as convergence of
\begin{equation*}
    \frac{(p+1)\ell_n}{a_n} \sum_{j=0}^{\lfloor \frac{a_n }{(p+1)\ell_n} r \rfloor - 1} \E \left( \left(\frac{\langle\int_0^1 h(u,\te_0,t) dt, f(u)  \rangle}{\iI(\te_0)} \xi_{p,j}^n + \frac 1{C(1,\te_0)} \ze_{p,j}^n(f) \right)^2 \middle|  \fF_{\frac{c_{n,p}(j)}{a_n}} \right)
\end{equation*}
towards
\begin{equation} 
    \begin{split}
        & \frac{p}{p+1} \int_0^r \left( \left(\frac{\langle\int_0^1 h(u,\te_0,t) dt, f(u)  \rangle}{\iI(\te_0)}\right)^2 \frac{ \left(\frac{\partial}{\partial \te}\mu(s,\te_0) \right)^2}{\mu(s,\te_0)} \right. \\
        & ~~~~~~~~~~~~+  \left.  2 \frac{\langle\int_0^1 h(u,\te_0,t) dt, f(u)  \rangle}{\iI(\te_0)C(1,\te_0)} \Phi(f) \frac{\partial}{\partial \te} \mu(s,\te_0) + \frac{\De(f)}{C(1,\te_0)^2} \mu(s,\te_0)\right) ds, \label{limit}
    \end{split}
\end{equation}
in probability as $n \to \infty$, the latter for any $r \in [0,1]$, and everything for a fixed $p$ and a fixed $f$. The proof is then finished as (\ref{limit}) equals $\frac p{p+1}\si^2(f)$ for $r=1$. Regarding the terms involving $\ze_{p,j}^n(f)$, note that we may replace them by $\Wze_{p,j}^n(f)$ for both claims because the same reasoning that was leading to (\ref{zeWze}) applies here as well.

For the fourth moments, note that we only need to prove uniform boundedness of $\E(|\xi_{p,j}^n|^4)$ and $\E(|\Wze_{p,j}^n(f)|^4)$. For the first one, recall basically from (\ref{xi1}) that we can write
\begin{equation*}
    \xi_{p,j}^n = \frac{1}{\sqrt{(p+1)\ell_n}} \left( \int_{\frac{c_{n,p}(j)}{a_n}}^{\frac{d_{n,p}(j)}{a_n}} U_n(t,\te_n) dN_n(t) - \E\left(\int_{\frac{c_{n,p}(j)}{a_n}}^{\frac{d_{n,p}(j)}{a_n}} U_n(t,\te_n) dN_n(t) \middle|\fF_{\frac{c_{n,p}(j)}{a_n}}\right)\right),
\end{equation*}
and (2.1.37) in \cite{jacpro2012} shows
\begin{equation*}
    \begin{split}
        & \E(|\xi_{p,j}^n|^4) \\
        & ~~~~\le K \frac{1}{((p+1)\ell_n)^2} \times \left\{ a_n \int_{\frac{c_{n,p}(j)}{a_n}}^{\frac{d_{n,p}(j)}{a_n}} U_n^4(t,\te_n) \mu(t,\te_n) dt + \left( a_n \int_{\frac{c_{n,p}(j)}{a_n}}^{\frac{d_{n,p}(j)}{a_n}} U^2_n(t,\te_n) \mu(t,\te_n) dt\right)^2 \right\} \\
        & ~~~~\le K \left( \frac{p}{(p+1)^2\ell_n} + \frac{p^2}{(p+1)^2} \right) \le K.
    \end{split}
\end{equation*}
Utilizing (\ref{fub}) we also have
\begin{equation*}
    \begin{split}
        & \E\left( \left|\Wze_{p,j}^n(f)\right|^4 \right) \\
        & ~~~~\le K \frac{1}{((p+1)\ell_n)^2} \times \left\{ \sum_{i \ge 1} \E \left(\left \langle k_n(u, \te_n, t_{i-1}, t_i) - L(u), f(u) \right \rangle^4 1_{\{\frac{c_{n,p}(j)}{a_n} < t_{i-1} \le \frac{d_{n,p}(j)}{a_n}\}} \right) \right. \\
        & ~~~~~~~~+ \left.\left( \sum_{i \ge 1} \E \left(\left \langle k_n(u, \te_n, t_{i-1}, t_i) - L(u), f(u) \right \rangle^2 1_{\{\frac{c_{n,p}(j)}{a_n} < t_{i-1} \le \frac{d_{n,p}(j)}{a_n}\}} \right) \right)^2 \right\}.
    \end{split}
\end{equation*}
By uniform boundedness of $\langle k_n(u, \te_n, t_{i-1}, t_i) - L(u), f(u) \rangle$ we then obtain
\begin{equation*}
    \E(|\Wze_{p,j}^n(f)|^4) \le K \left( \frac{p}{(p+1)^2\ell_n} + \frac{p^2}{(p+1)^2} \right) \le K
\end{equation*}
as above.

We will finally discuss the convergence towards (\ref{limit}), but we will restrict ourselves to the case $r=1$ as the proof is completely similar in the general setting. First, recalling the rules established in the beginning of the proof of Lemma \ref{lem3},
\begin{equation*}
    \begin{split}
        & \E \left(|\xi_{p,j}^n|^2 \middle|  \fF_{\frac{c_{n,p}(j)}{a_n}} \right) = \frac{1}{(p+1)\ell_n} \Var\left( \int_{\frac{c_{n,p}(j)}{a_n}}^{\frac{d_{n,p}(j)}{a_n}} U_n(s,\te_n) dN_n(s)\middle|  \fF_{\frac{c_{n,p}(j)}{a_n}}\right) \\
        & ~~= \frac{a_n}{(p+1)\ell_n} \int_{\frac{c_{n,p}(j)}{a_n}}^{\frac{d_{n,p}(j)}{a_n}} U_n^2(s,\te_n) \mu(s,\te_n) ds = \frac{a_n}{(p+1)\ell_n} \int_{\frac{c_{n,p}(j)}{a_n}}^{\frac{d_{n,p}(j)}{a_n}} \frac{ \left(\frac{\partial}{\partial \te}\mu(s,\te_n) \right)^2}{\mu(s,\te_n)} ds.
    \end{split}
\end{equation*}
Also, using (\ref{fub}) and (\ref{condhilf}) again,
\begin{equation*}
    \begin{split}
        & \E \left(|\Wze_{p,j}^n(f)|^2 \middle|  \fF_{\frac{c_{n,p}(j)}{a_n}} \right) = \frac{1}{(p+1)\ell_n} \times \\
        & ~~~~~~\sum_{i \ge 1} \E\left( \E\left(\left \langle k_n(u, \te_n, t_{i-1}, t_i) - L(u), f(u) \right \rangle^2 \middle| \fF_{t_{i-1}} \right) 1_{\{\frac{c_{n,p}(j)}{a_n} < t_{i-1} \le \frac{d_{n,p}(j)}{a_n}\}} \middle|  \fF_{\frac{c_{n,p}(j)}{a_n}} \right),
    \end{split}
\end{equation*}
and we also have
\begin{equation*}
    \E\left(\left \langle k_n(u, \te_n, t_{i-1}, t_i) - L(u), f(u) \right \rangle^2 \middle| \fF_{t_{i-1}} \right) = \E\left(\left \langle \exp(-uX) - L(u), f(u) \right \rangle^2 \right)
\end{equation*}
with $X \sim \exp(1)$ again. A use of Fubini's theorem proves
\begin{equation*}
    \begin{split}
        &\E\left(\left \langle \exp(-uX) - L(u), f(u) \right \rangle^2\right) \\
        & ~~= \int_0^\infty \int_0^\infty \int_0^\infty \left(\exp(-ux) - \frac 1{1+u} \right)  \left(\exp(-vx) - \frac 1{1+v} \right) \\
        & ~~~~~~~~~~~~~~~~~~~~~~~~~~~~~~~~~~~~~~~~~~~~~\exp(-x) dx \be(u) f(u)du \be(v)f(v) dv \\
        & ~~= \int_0^\infty \int_0^\infty \left(\frac 1{1+u+v} -  \frac 1{(1+u)(1+v)} \right) \be(u) f(u) \be(v) f(v) du dv = \De(f).
    \end{split}
\end{equation*}
Hence,
\begin{equation*}
    \begin{split}
        \E \left(|\Wze_{p,j}^n(f)|^2 \middle|  \fF_{\frac{c_{n,p}(j)}{a_n}} \right) & = \frac{1}{(p+1)\ell_n} \De(f)  \E\left( \left(N\left(\frac{d_{n,p}(j)}{a_n}\right) - N\left(\frac{c_{n,p}(j)}{a_n}\right) \right) \middle|  \fF_{\frac{c_{n,p}(j)}{a_n}}\right) \\
                                                                                    & = \frac{a_n}{(p+1)\ell_n} \De(f) \int_{c_{n,p}(j)}^{d_{n,p}(j)} \mu(s,\te_n) ds.
    \end{split}
\end{equation*}

We finally have to discuss \(\E \left(\xi_{p,j}^n \Wze_{p,j}^n(f) \middle|  \fF_{\frac{c_{n,p}(j)}{a_n}} \right)\), which becomes
\begin{equation}  \label{step2}
    \begin{split}
        & \frac{1}{(p+1)\ell_n} \sum_{i_1,i_2 \ge 1} \E \left( U_n(t_{i_1-1},\te_n) \left \langle k_n(u, \te_n, t_{i_2-1}, t_{i_2}) - L(u), f(u) \right \rangle \times \right. \\
        & ~~~~~~~~~~~~~~~~~~~~~~~~~~~~~~~~~~~~~~~~~~~\left. 1_{\{\frac{c_{n,p}(j)}{a_n} < t_{i_1-1}, t_{i_2-1} \le \frac{d_{n,p}(j)}{a_n}\}} \middle|  \fF_{\frac{c_{n,p}(j)}{a_n}} \right)
    \end{split}
\end{equation}
after applying (\ref{fub}). The same reasoning also proves that the conditional expectation above vanishes for $i_1 \le i_2$. On the other hand, for any fixed $i_2$ we obtain
\begin{equation*}
    \begin{split}
        & \sum_{i_1 > i_2} \E \left( U_n(t_{i_1-1},\te_n) \left \langle k_n(u, \te_n, t_{i_2-1}, t_{i_2}) - L(u), f(u) \right \rangle \times \right. \\
        &~~~~~~~~~~~~~~~~~~~~~~~~~~~~~~~~~~~~~~~~~~~ \left. 1_{\{\frac{c_{n,p}(j)}{a_n} < t_{i_1-1}, t_{i_2-1} \le \frac{d_{n,p}(j)}{a_n}\}} \middle|  \fF_{\frac{c_{n,p}(j)}{a_n}} \right) \\
        & ~~= \E \left( U_n(t_{i_2},\te_n) \left \langle k_n(u, \te_n, t_{i_2-1}, t_{i_2}) - L(u), f(u) \right \rangle 1_{\{\frac{c_{n,p}(j)}{a_n} < t_{i_2-1} < t_{i_2 }\le \frac{d_{n,p}(j)}{a_n}\}} \middle|  \fF_{\frac{c_{n,p}(j)}{a_n}} \right) \\
        & ~~~~~~~~+ \E \left(\sum_{i_1 > i_2+1 } \E  \left[ U_n(t_{i_1-1},\te_n) 1_{\{t_{i_2} < t_{i_1-1} \le \frac{d_{n,p}(j)}{a_n}\}} \middle| \fF_{t_{i_2}} \right)\times \right. \\
        & ~~~~~~~~~~~~~~~~~~~~~\left. \left \langle k_n(u, \te_n, t_{i_2-1}, t_{i_2}) - L(u), f(u) \right \rangle 1_{\{\frac{c_{n,p}(j)}{a_n} < t_{i_2-1} < t_{i_2 }\le \frac{d_{n,p}(j)}{a_n}\}}  \middle|  \fF_{\frac{c_{n,p}(j)}{a_n}} \right].
    \end{split}
\end{equation*}
From (\ref{intU}) the term above thus becomes
\begin{equation*}
    \begin{split}
        & \E \left(1_{\{\frac{c_{n,p}(j)}{a_n} < t_{i_2-1} < t_{i_2 }\le \frac{d_{n,p}(j)}{a_n}\}} \left \langle k_n(u, \te_n, t_{i_2-1}, t_{i_2}) - L(u), f(u) \right \rangle \right.\\
        & ~~~~~~~~ \left. \left(U_n(t_{i_2},\te_n) + \int_{t_{i_2}}^{\frac{d_{n,p}(j)}{a_n}} a_n \frac{\partial}{\partial \te} \mu(r,\te_n)dr \right) \middle|  \fF_{\frac{c_{n,p}(j)}{a_n}} \right) \\
        & ~~= \E \left( \E\left(1_{\{{t_{i_2-1}} < t_{i_2}\le \frac{d_{n,p}(j)}{a_n}\}} \left \langle k_n(u, \te_n, t_{i_2-1}, t_{i_2}) - L(u), f(u) \right \rangle \right. \right.\\
        & ~~~~~~~~ \left. \left. \left(U_n(t_{i_2},\te_n) + \int_{t_{i_2}}^{\frac{d_{n,p}(j)}{a_n}} a_n \frac{\partial}{\partial \te} \mu(r,\te_n)dr \right) \middle|  \fF_{t_{i_2-1}} \right) 1_{\{\frac{c_{n,p}(j)}{a_n} < t_{i_2-1} \le \frac{d_{n,p}(j)}{a_n}\}} \middle|  \fF_{\frac{c_{n,p}(j)}{a_n}} \right).
    \end{split}
\end{equation*}
We apply (\ref{density}). Hence, setting
\begin{equation*}
    \begin{split}
        & \al_n(t,u) = \int_{t}^{\frac{d_{n,p}(j)}{a_n}} \left(\exp\left(-u \int_{t}^s a_n \mu(r,\te_n) dr  \right) - L(u) \right) \times \\
        & ~~~~\left (U_n(s,\te_n) + \int_{s}^{\frac{d_{n,p}(j)}{a_n}} a_n \frac{\partial}{\partial \te} \mu(r,\te_n)dr \right) \exp\left( -\int_{t}^s a_n \mu(r,\te_n) dr  \right) a_n \mu(s,\te_n) ds
    \end{split}
\end{equation*}
and utilizing Fubini's theorem, (\ref{step2}) becomes
\begin{equation*}
    \begin{split}
        & \frac{1}{(p+1)\ell_n}\left \langle \E \left(\sum_{i_2 \ge 1}\al_n(t_{i_2-1},u)  1_{\{\frac{c_{n,p}(j)}{a_n} < t_{i_2-1} \le \frac{d_{n,p}(j)}{a_n}\}}\middle|  \fF_{\frac{c_{n,p}(j)}{a_n}} \right), f(u) \right \rangle \\
        & ~~= \frac{1}{(p+1)\ell_n}\left \langle \int_{\frac{c_{n,p}(j)}{a_n}}^{\frac{d_{n,p}(j)}{a_n}} a_n \al_n(t,u) \mu(t,\te_n) dt, f(u) \right \rangle.
    \end{split}
\end{equation*}
Regarding the inner integral, we obtain
\begin{equation*}
    \begin{split}
        &\int_{\frac{c_{n,p}(j)}{a_n}}^{\frac{d_{n,p}(j)}{a_n}} a_n \al_n(t,u) \mu(t,\te_n) dt = \int_{\frac{c_{n,p}(j)}{a_n}}^{\frac{d_{n,p}(j)}{a_n}} \frac 1{1+u} \left(U_n(s,\te_n) + \int_{s}^{\frac{d_{n,p}(j)}{a_n}} a_n \frac{\partial}{\partial \te} \mu(r,\te_n)dr \right)  \times \\
        & ~~\left( \exp\left(-\int_{\frac{c_{n,p}(j)}{a_n}}^s a_n \mu(r,\te_n) dr  \right) - \exp\left(-\int_{\frac{c_{n,p}(j)}{a_n}}^s (1+u) a_n \mu(r,\te_n) dr  \right)  \right)  a_n \mu(s,\te_n) ds
    \end{split}
\end{equation*}
from Fubini's theorem first. Evaluating the latter integral is hard, but several applications of Condition \ref{condmain}(c) as well as $\ell_n$ being of the order $a_n^\vr$, $0 < \vr < 1/2$, allow to replace $\mu(s,\te_n)$ by $\mu(\frac{c_{n,p}(j)}{a_n},\te_n)$ with an error or small order, and similarly for the derivatives. Also, the summand involving $U_n(s,\te_n)$ is of small order as well. Hence the approximation
\begin{equation*}
    \begin{split}
        &\int_{\frac{c_{n,p}(j)}{a_n}}^{\frac{d_{n,p}(j)}{a_n}} a_n \al_n(t,u) \mu(t,\te_n) dt = \int_{\frac{c_{n,p}(j)}{a_n}}^{\frac{d_{n,p}(j)}{a_n}} \frac 1{1+u}  a_n^2 \frac{\partial}{\partial \te} \mu(\frac{c_{n,p}(j)}{a_n},\te_n) \left(\frac{d_{n,p}(j)}{a_n}-s\right)  \times \\
        & ~~\left(\exp\left(-a_n \mu(\frac{c_{n,p}(j)}{a_n},\te_n) \left(s-\frac{c_{n,p}(j)}{a_n}\right)  \right) \right.\\
        & ~~~~~~~~-\left.  \exp\left(-a_n(u+1) \mu(\frac{c_{n,p}(j)}{a_n},\te_n) \left(s-\frac{c_{n,p}(j)}{a_n}\right)\right)  \right) \mu(\frac{c_{n,p}(j)}{a_n},\te_n) ds \left(1+o(1)\right)
    \end{split}
\end{equation*}
with a small order term uniform in $\te$ and $u$ holds (note again that $p$ is fixed). To evaluate the latter integral we use
\begin{equation*}
    \int_a^b \exp(-K(s-a)) (b-s) ds = \frac{(b-a)K - (1-\exp(-(b-a)K))}{K^2}
\end{equation*}
for generic positive $a, b$ and $K$, and we see that the latter fraction to first order equals $\frac{b-a}K$ in our setting. To summarize,
\begin{equation*}
    \int_{\frac{c_{n,p}(j)}{a_n}}^{\frac{d_{n,p}(j)}{a_n}} a_n \al_n(t,u) \mu(t,\te_n) dt = \frac{1}{1+u} \left(1-\frac 1{1+u} \right) p\ell_n \frac{\partial}{\partial \te} \mu(\frac{c_{n,p}(j)}{a_n},\te_n) \left(1+o(1)\right)
\end{equation*}
with a small order term as above. Using Condition \ref{condmain}(c) again we can then write (\ref{step2}) as
\begin{equation*}
    \frac{p}{p+1} \Phi(f) \frac{\partial}{\partial \te} \mu(\frac{c_{n,p}(j)}{a_n},\te_n) \left(1+o(1)\right) =  \frac{a_n}{(p+1)\ell_n} \Phi(f) \int_{\frac{c_{n,p}(j)}{a_n}}^{\frac{d_{n,p}(j)}{a_n}} \frac{\partial}{\partial \te}\mu(s,\te_n) ds \left(1+o(1)\right).
\end{equation*}
The result is now easy to conclude from Condition \ref{condmain}(c) again. \qed

%
%

\bibhang=1.7pc
\bibsep=2pt
\fontsize{9}{14pt plus.8pt minus .6pt}\selectfont
\renewcommand\bibname{\large \bf References}
\expandafter\ifx\csname
natexlab\endcsname\relax\def\natexlab#1{#1}\fi
\expandafter\ifx\csname url\endcsname\relax
\def\url#1{\texttt{#1}}\fi
\expandafter\ifx\csname urlprefix\endcsname\relax\def\urlprefix{URL}\fi

\bibliographystyle{chicago}      
\bibliography{references}   

\vskip .65cm
\noindent
Christian-Albrechts-Universit\"at zu Kiel, Mathematisches Seminar\newline
GEOMAR Helmholtz Centre for Ocean Research
\vskip 2pt
\noindent
E-mail: jkling@geomar.de
\vskip 2pt

\noindent
Christian-Albrechts-Universit\"at zu Kiel, Mathematisches Seminar
\vskip 2pt
\noindent
E-mail: vetter@math.uni-kiel.de

\end{document}